\documentclass[11pt]{article}
\usepackage{amssymb,amsmath,amsthm,amsfonts,color,booktabs,subfigure,float}
\usepackage[dvips]{graphicx}
\DeclareMathOperator{\rank}{rank}
\DeclareMathOperator{\diag}{diag}

\DeclareMathOperator{\argmin}{argmin}

\theoremstyle{plain}
\newtheorem{theorem}{Theorem}[section]
\newtheorem{lemma}{Lemma}[section]

\numberwithin{equation}{section}

\theoremstyle{definition}

\newtheorem{remark}{Remark}[section]

\def\dref#1{(\ref{#1})}
\def\dfrac{\displaystyle\frac}
\def\crr{\cr\noalign{\vskip2mm}}
\def\disp{\displaystyle}
\def\H{{\cal H}}
\def\A{{\cal A}}

\newcommand{\N}{{\mathbb N}}
\newcommand{\K}{{\mathbb K}}

\begin{document}

\title{{\bf Simultaneous Identification of  Coefficient and Initial State
for One-Dimensional Heat Equation from Boundary Control and
Measurement}}
\author{Zhi-Xue Zhao$^{a,b}$ M.K. Banda$^b$,  and  Bao-Zhu Guo$^{c,d}$\footnote{
Corresponding author. Email: bzguo@iss.ac.cn}\\
{\it $^a$School of Mathematical Sciences, }\\{\it Tianjin Normal University, Tianjin 300387, China}\\
{\it $^b$Department of Mathematics and Applied Mathematics,}\\{\it
University of Pretoria,
 Pretoria 0002, South Africa}\\
{\it $^c$Academy of Mathematics and Systems Science, }\\{\it Academia Sinica, Beijing 100190, China, }\\
{\it $^d$School of Computer Science and Applied Mathematics, }\\{\it
University of the Witwatersrand, Johannesburg, South Africa} }
\date{}
\maketitle
\begin{center}
\begin{abstract}

In this paper, we consider simultaneous reconstruction of the
diffusion coefficient and initial state for a one-dimensional heat
equation through boundary control and measurement. The boundary
measurement is known to make the system exactly observable, and both
coefficient and initial state are shown to be identifiable by this
measurement. By a Dirichlet series representation for observation,
we can transform the problem into an inverse process of
reconstruction of the spectrum and coefficients for Dirichlet series
in terms of observation. This happens to be the reconstruction of
spectral data for an exponential sequence with measurement error.
This enables us to develop an algorithm based on the matrix pencil
method in signal analysis. An error analysis is made for the
proposed method. The numerical simulations are presented to verify
the proposed algorithm.

\vspace{0.3cm}


\vspace{0.3cm}


\end{abstract}

\end{center}

\section{Introduction}
\setcounter{equation}{0}

It is recognized that many industrial controls are temperature
control. The inverse heat conduction problem (IHCP) is one of the
important control problems in science and engineering. Such kinds of
problems usually arise in modeling and process control with heat
propagation in thermophysics, chemical engineering, and many other
industrial and engineering applications. In the last decades, there
are various class  of IHCPs having been investigated ranging from
recovery of boundary heat flux \cite{Wang1}; estimation of medium
parameters such as thermal conductivity coefficient
\cite{Gutman,Lorenzi1} and radiative coefficient
\cite{Choulli1,Deng1,Yamamoto1}; recovery of spatial distribution of
heat sources \cite{Ma,Zheng1}; and reconstruction of initial state
distributions \cite{Yamamoto1}. For many other aspects including
numerical solutions of  inverse problems for PDEs, we refer to
monographs \cite{Isakov1} and \cite{Kirsch1}.

Most of the existing works, however, are devoted to single parameter
identification. The simultaneous reconstruction of more than one
different coefficients within a dynamical framework has not been
sufficiently investigated, for which, to the best of our knowledge,
only a few studies are available. In \cite{Benabdallah1}, the
uniqueness and stability of determining both diffusion coefficient
and initial condition from a measurement of the solution at a
positive time and on an arbitrary part of the boundary for a heat
equation are discussed. In \cite{Choulli1} uniqueness and stability
estimate of an inverse problem for a parabolic equation, where
simultaneously  determination of heat radiative coefficient, the
initial temperature, and a boundary coefficient from a temperature
distribution measured at a positive moment is considered. In
\cite{Wang1}, a numerical method is presented to determine both
initial value and  boundary value at one end from the discrete
observation data at the other end. Given a measurement of
temperature at a single instant of time and measurement of
temperature in a subregion of the physical domain, \cite{Yamamoto1}
investigates stability and numerical reconstruction of initial
temperature and radiative coefficient for a heat conductive system.

In these  works aforementioned, most of the results about uniqueness
and stability are based on the Carleman estimates for which
\cite{Yamamoto2} presents a brief review on the application of the
Carleman estimates to inverse problems for parabolic equations. To
cope with the   ill-posed nature of inverse problems, optimization
methods and regularization techniques together with many other
numerical methods such as finite difference method,  finite element
method, and   boundary element method are generally applied in
literature.

In addition  to the numerical methods used in literature cited
above, the inverse spectral theory is also considered as an
important tool in the study of inverse problems
\cite{Levitan,Poschel}. Solutions of inverse spectral problems
generate certain geometric and physical parameters from the spectral
data, like shape of the region, coefficients of conductivity, etc.
 A number of classical identifiability results are based
on the inverse spectral theory, see, for instance,
\cite{Murayama,Pierce,Suzuki1,Suzuki2}. In \cite{Pierce}, the unique
determination of eigenvalues and coefficients under certain
conditions for a  parabolic equation is considered by
Gel'fand-Levitan theory. Some uniqueness results on the simultaneous
identification of coefficients and initial values for parabolic
equations are given in \cite{Murayama,Suzuki1,Suzuki2}. However,
most of the identifiability results require that the initial value
can not be orthogonal to any of the eigenvectors. This restrictive
condition is actually unverifiable in practice since the initial
value is also unknown. Some other uniqueness results on the
determination of constant coefficients are discussed in
\cite{Kitamura,Nakagiri}. But no numerical identification algorithm
is attempted in these theoretical papers.

In this paper, we are concerned with reconstruction of the diffusion
coefficient and initial state for a one-dimensional heat conduction
equation in a homogeneous bar of unit length, which is described by
\begin{equation}\label{1.1}
\left\{\begin{array}{ll} u_{t}(x,t)=\alpha u_{xx}(x,t),& 0<x<1,\;
t>0,\crr\disp \alpha u_x(0,t)=f(t),\;\; u_x(1,t)=0,& t\geq
0,\crr\disp y(t)=u(0,t),& t\ge 0,\crr\disp  u(x,0)=u_0(x),& 0\leq
x\leq 1,
\end{array}\right.
\end{equation}
where $x$ represents  the position, $t$ the time.
$\alpha\geq\alpha_0>0$ is an unknown constant that represents the
thermal diffusivity, $u_0(x)$ is the unknown initial temperature
distribution, both of them need to be identified. Here we do not
impose any restriction on the initial value other than boundedness.
The function $f(t)$ is the Neumann boundary control (input) which
represents the heat flux through the left end of the bar, and $y(t)$
is the boundary temperature measurement. Sometimes we write the
solution of \dref{1.1} as $u=u(x,t;f,u_0)$ to denote its dependence
on $f(t)$ and $u_0(x)$.

Let $\H=L^2(0,1)$ with the usual inner product $\langle\cdot,
\cdot\rangle$ and inner product induced norm $\|\cdot\|$. Define the
operator $\A:\; D(\A)(\subset \H)\mapsto \H$ by
\begin{equation}\label{2.1}
\left\{
\begin{array}{l}
[\A \psi](x)=-\alpha \psi^{\prime\prime}(x),\crr\disp
D(\A)=\{\psi\in H^2(0,1)\,|\,\psi^\prime(0)=\psi^\prime(1)=0\}.
\end{array}
\right.
\end{equation}
It is well known that this defined operator $\A$ is positive
semidefinite in $\H$. The eigenvalues $\{\lambda_n\}$ are given by
\begin{equation}\label{2.2}
\lambda_n=\alpha n^2\pi^2,\; n=0,1,2,\dots,
\end{equation}
and the corresponding eigenfunctions $\{\phi_n(x)\}_{n=0}^\infty$
are given by
\begin{equation}\label{2.3}
\phi_0(x)=1,\; \phi_n(x)=\sqrt{2}\cos n\pi x,\;  n=1,2,\dots.
\end{equation}
Denote
$$\N_0=\N \cup \{0\}.
$$
It is well known that $\{\phi_n(x)\}_{n\in \N_0}$ forms an
orthonormal basis for $\H$.

Set
\begin{equation}\label{2.7}
\begin{array}{l}
\disp G(t,x,y)=\sum_{n=1}^\infty e^{-\lambda_n
t}\phi_n(x)\phi_n(y)+1,\crr \disp A_0(x)=\int_0^1 u_0(x)dx,\crr
\disp A_n(x)=\langle u_0,\phi_n\rangle
\phi_n(x)=2\left(\disp\int_0^1u_0(y)\cos n\pi y \; dy\right)\cos
n\pi x,\; n\in \N.
\end{array}
\end{equation}
A standard analysis (\cite{Chang}) shows that the solution of system
\dref{1.1} can be represented by
\begin{equation}\label{2.10}
u(x,t;f,u_0)=\disp\sum_{n=0}^\infty A_n(x)e^{-\lambda_n
t}-\disp\int_0^t G(t-s,x,0)f(s) \; ds, \; 0\leq x\leq 1,\;t>0.
\end{equation}
Therefore, the boundary observation $y(t)$ takes the form
\begin{equation}\label{2.11}
y(t)=u(0,t;f,u_0)=\disp\sum_{n=0}^\infty A_n(0)e^{-\lambda_n
t}-\disp\int_0^t G(t-s,0,0)f(s) \; ds,\forall\;  t>0.
\end{equation}

The inverse problem that we consider in this paper can be described
as follows:

\vspace{0.15cm}

\noindent{\bf Inverse Problem:} {\it Given $f(t)$ and $y(t)$  in a
finite time interval $t\in [0,T]$, determine $\alpha$ and $u_0(x)$
simultaneously.}

\vspace{0.15cm}

Let us briefly explain the main idea of this paper, which is
inspired by an  idea of  \cite{Guo1}. By \dref{2.11}, the output
$y(t)$ of system \dref{1.1} is separated into  two  parts
$u(0,t;0,u_0)$ and $u(0,t;f,0)$, where the former is determined by
the initial state only and the latter by the control. The first part
$u(0,t;0,u_0)$ admits a Dirichlet series representation, which means
that it can be determined by its restriction on any finite interval.
By choosing the control $f(t)$ appropriately, we can design an
algorithm to estimate the unknown coefficients in the Dirichlet
series. The part of the output that has been determined by the
initial value, $u(0,t;0,u_0)$, can be substituted in equation
\dref{2.11} of the output such that the coefficient identification
of $\alpha$ is equivalently transformed into the case with zero
initial state (see section 2 and 3.3 for details). After estimating
the coefficient $\alpha$, the remaining problem is a single
reconstruction of the initial state.

The rest of the paper is organized as follows. Section 2 is devoted
to simultaneous identifiability of the coefficient and initial value
based on the Dirichlet series theory. The identification algorithm
based on the matrix pencil method is introduced in section 3. In
section 4, the error analysis of the matrix pencil method to the
infinite spectral estimation problem is obtained. A numerical
simulation is presented in section 5 to show the validity of the
algorithm introduced in section 3.

\section{Identifiability}

Since we want to reconstruct simultaneously the diffusion
coefficient $\alpha$ and the initial state $u_0(x)$ of system
\dref{1.1} from the boundary control $f(t)$ and observation
$y(t)=u(0,t)$,  we need first to make sure that the data
$\{f(t),u(0,t)\}$ is sufficient to determine $\alpha$ and $u_0(x)$
uniquely. This is the identifiability from system control point of
view.

Suppose that $T_2>T_1>0$ are two arbitrary positive numbers,  and
the boundary control function $f(t)$ is chosen to be zero during the
time interval $[0,T_2]$. In this case, it is deduced from
\dref{2.11} that the boundary observation is
\begin{equation}\label{3.1}
y(t)\triangleq u(0,t;0,u_0)=\sum_{n=0}^\infty C_ne^{-\lambda_n t},
\; \forall\; t\in [T_1,T_2],
\end{equation}
where
\begin{equation}\label{3.2}
\begin{array}{l}
C_0=A_0(0)=\disp\int_0^1 u_0(x)dx,\crr
C_n=A_n(0)=2\left(\disp\int_0^1u_0(y)\cos n\pi ydy\right),\; n\in
\N.
\end{array}
\end{equation}

\begin{theorem}\label{Th3.1}
Let $0\leq T_1<T_2<\infty$, $u_0 \in L^2(0,1)$ and let $\lambda_n$
and $C_n$ be defined as in \dref{2.2} and \dref{3.2}, respectively.
Then the set $\{(C_k,\lambda_k)\;|\; C_k\neq 0\}_{k\in\N_0}$ in
\dref{3.1} can be uniquely determined by the observation data
$\{y(t)|\; t\in[T_1,T_2]\}$.
\end{theorem}
\begin{proof} Since
$u_0(x)$ is unknown, it is not clear whether $C_n\neq 0$ for any
$n\in \N_0$. Define the set $\K\subset \N_0$, which is unknown as
well and satisfies
\begin{equation}\label{3.3}
\left\{\begin{array}{ll} C_k\neq 0,&k\in \K,\\
C_k=0,&k\notin \K.
\end{array}
\right.
\end{equation}
Then \dref{3.1} can be re-written as
\begin{equation}\label{3.4}
y(t)=\sum_{k\in\K} C_ke^{-\lambda_k t},\; \forall\;  t\in [T_1,T_2].
\end{equation}
The proof is accomplished by  two steps.

{\it  Step 1:  $\{(C_k,\lambda_k)\}_{k\in\K}$ can be uniquely
determined by infinite-time observation $\{y(t)|\; t\in
(0,\infty)\}$.}

Actually, since
$$
\sum_{n=0}^\infty |C_n|^2\leq 2\|u_0\|_{L^2(0,1)}^2<+\infty,
$$
it follows that  $\sup_{n\ge 0}|C_n|<\infty$. Since
$\lambda_n=\alpha n^2\pi^2$, the series \dref{3.1} converges
uniformly in $t$ over $(0,+\infty)$. Apply the Laplace transform to
\dref{3.1} to obtain
\begin{equation}\label{3.5}
\hat{y}(s)=\sum_{n=0}^\infty \dfrac{C_n}{s+\lambda_n}=\sum_{k\in\K}
\dfrac{C_k}{s+\lambda_k},
\end{equation}
where $\hat{}$ denotes the Laplace transform. It can be seen  from
\dref{3.5} that $-\lambda_k$ is a  pole of $\hat{y}(s)$ and $C_k$ is
the residue of $\hat{y}(s)$ at $-\lambda_k$ for $k\in\K$. By the
uniqueness of the Laplace transform, $\{(C_k,\lambda_k)\}_{k\in\K}$
is uniquely determined by $\{y(t)|\; t\in(0,\infty)\}$.

 {\it  Step 2: $\{(C_k, \lambda_k)\}_{k\in\K}$ can be uniquely
determined by finite-time observation $\{y(t)|\;  t\in
[T_1,T_2]\}$.}

By {\it step 1},  we only need to show that the observation $y(t)$
in \dref{3.1} for all $t>0$ can be uniquely determined by its
restriction on $I_1=[T_1,T_2]$, or in other words, $y(t)=0$ for
$t\in [T_1,T_2]$ in \dref{3.1} implies that $y(t)=0$ for all $t>0$.
But this is  obvious because $y(t)$ is an analytic function in
$t>0$. This completes the proof of the theorem.
\end{proof}

\begin{theorem}\label{Th3.2} Let $0\leq T_1<T_2<T_3\leq \infty$,
$u_0 \in L^2(0,1)$ and let $\lambda_n$ and $C_n$ be defined as in
\dref{2.2} and \dref{3.2}, respectively. The control function $f(t)$
satisfies
\begin{equation}\label{3.8}
\left\{
\begin{array}{ll}
f(t)=0,\;&\mbox{for\;\;} t\in[0,T_2)\\
f(t)\neq 0,\;&\mbox{for almost all\;\;} t\in [T_2,T_3]
\end{array}
\right.
\end{equation}
and the corresponding observation data is
$\{y(t)=u(0,t;f,u_0)|\;t\in [T_1,T_3]\}$. Then the diffusion
coefficient $\alpha$ and the initial state $u_0(x)$ in system
\dref{1.1} can be uniquely determined by the observation
$\{y(t)|\;t\in [T_1,T_3]\}$.
\end{theorem}
\begin{proof}
By \dref{2.11}, for $t\in[T_2,T_3]$,
\begin{equation}\label{3.9}
y(t)=\sum_{k\in\K} C_ke^{-\lambda_k t}-\int_{T_2}^t G(t-s,0,0)f(s)
\; ds.
\end{equation}
Set
\begin{equation}\label{3.12}
\widetilde{y}(t)=\sum_{k\in\K} C_ke^{-\lambda_k (t+T_2)}-y(t+T_2).\;
t\in[0,T_3-T_2],
\end{equation}
Then \dref{3.9} takes the following form
\begin{equation}\label{3.14}
\widetilde{y}(t)=\int_0^t G(t-s,0,0)f(s+T_2) \; ds,\; t\in
[0,T_3-T_2].
\end{equation}
Since $f(t)\neq 0$ for almost all $t\in [T_2,T_3]$, the integral
equation \dref{3.14} has a unique solution $G(t,0,0)$ \cite[Theorem
151, p.324]{Titchmarsh}, which means that
\begin{equation}\label{3.14.2}
G(t,0,0)=2\sum_{n=1}^\infty e^{-\lambda_n t}+1, \; t\in [0,T_3-T_2]
\end{equation}
can be uniquely determined by $\widetilde{y}(t),\;t\in [0,T_3-T_2]$.
By Theorem \ref{Th3.1}, $\{(C_k,\lambda_k)\}_{k\in \K}$ can be
determined from the observation $\{y(t)|\; t\in[T_1,T_2]\}$, which
shows, from \dref{3.12}, that $\widetilde{y}(t),\;t\in [0,T_3-T_2]$
can be obtained from   $\{y(t)|\; t\in [T_1,T_3]\}$.

Since all the coefficients of the exponents in \dref{3.14.2} are
nonzero, by Theorem \ref{Th3.1} again, $\{\lambda_n\}_{n\in\N}$ can
be uniquely determined by $\{G(t,0,0)|\; t\in [0,T_3-T_2]\}$. Hence,
the exponents $\{\lambda_n\}_{n\in\N}$ are uniquely determined by
 $\{y(t)|\; t\in [T_1,T_3]\}$, and then the
diffusion coefficient $\alpha$ can be obtained from
$\lambda_n=\alpha n^2\pi^2$. This proves the identifiability of
$\alpha$.

Given $\alpha$ is known,  and since $\lambda_i\neq \lambda_j$ for
$i\neq j$, we can also determine the set $\K$ by comparing
$\{\lambda_k\}_{k\in\K}$ with $\{\lambda_n=\alpha
n^2\pi^2\}_{n\in\N_0}$, where $\K$ is defined through \dref{3.3}.
The initial value $u_0(x)$  is therefore uniquely determined by
\begin{equation}\label{3.16}
u_0(x)=\sum_{n=0}^\infty \langle u_0,\phi_n\rangle
\phi_n(x)=\sum_{k\in\K} C_k\cos k\pi x.
\end{equation}
This completes the proof of the theorem.
\end{proof}

\begin{remark}\label{Re2.1}
There are many  papers studying the simultaneous identifiability of
parameters and initial values for  parabolic equations, see, for
instance, \cite{Murayama,Suzuki1,Suzuki2}. However, most of the
identifiability results require that the initial value should be a
generating element (see \cite{Suzuki1}) with respect to the system
operator $\A$, that is,
\begin{equation}\label{3.14.4}
\langle u_0,\phi_n\rangle\neq 0,\; \mbox{for any\;} n\in\N_0.
\end{equation}
But this condition is unverifiable because  the initial value
$u_0(x)$ is also unknown. In Theorem \ref{Th3.2}, this restrictive
condition on the initial value is removed by designing the control
signal properly. The simplest practically implementable control that
satisfies \dref{3.8} is
\begin{equation}\label{3.14.3}
f(t)=\left\{
\begin{array}{ll}
0,& t\in [0,T_2),\\
1,& t\in [T_2,T_3].
\end{array}
\right.
\end{equation}
which is used in the numerical identification algorithm in section
3.
\end{remark}

\begin{remark}\label{Re2.2}
It is known that persistently exciting (PE) condition plays a
crucial role in adaptive parameter identification to ensure the
convergence, see e.g., \cite{Orlov,Smyshlyaev}. It seems that the
control signal \dref{3.8} in Theorem \ref{Th3.2} is similar to that
in \cite{Smyshlyaev}, where the nonzero constant input is proved to
satisfy the PE condition. Although the method in \cite{Smyshlyaev}
is online identification (for different coefficients) whereas here
it is offline, the condition \dref{3.8} is also to excite
persistently the plant behavior. To illustrate the identifiability
analysis more clearly, we give a block diagram in Figure
\ref{Fig-1.1}.
\begin{figure}[h!]
\centering
\includegraphics[width=2.6in]{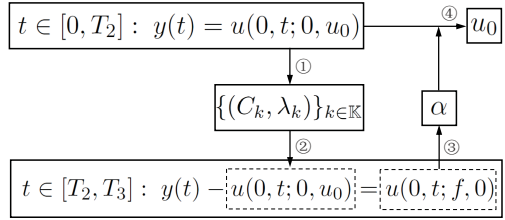}
\caption{Block diagram of identifiability analysis}\label{Fig-1.1}
\end{figure}
\end{remark}

\section{Numerical computation method}

It is seen clearly from previous section that the  key point for
identification purposes is to recover the spectrum-coefficient data
$\left\{(C_n,\lambda_n)\right\}_{n\in\N_0}$ from $\{y(t)|\; t\in
[T_1,T_2]\}$ by the Dirichlet series representation:
\begin{equation}\label{4.01}
y(t)=\sum_{n=0}^\infty C_n e^{-\lambda_n t},\;  t\in [T_1,T_2].
\end{equation}
The difficulty is that  there may exist infinitely many $C_n\neq 0$
in \dref{4.01}. In this section, we use the matrix pencil method to
extract some of the $\{(C_n,\lambda_n)\}$ from the sum of the first
$M$ terms of the infinite series \dref{4.01}, and treat the
remainder terms just as a measurement error.

\subsection{Finite dimensional approximation of spectral estimation}

Suppose $M\in\N$ and split the series in \dref{4.01} into two parts:
\begin{equation}\label{4.02}
y(t)=\sum_{n=0}^{M-1} C_n e^{-\lambda_n t}+\sum_{n=M}^\infty C_n
e^{-\lambda_n t},\; \forall\; t\in  [T_1,T_2].
\end{equation}
Denote the second series in \dref{4.02} as
\begin{equation}\label{4.03}
e(M,t)=\sum_{n=M}^\infty C_n e^{-\lambda_n t},\; \forall\; t\in
[T_1,T_2].
\end{equation}
 Theorem \ref{Th4.1} gives  a  bound of $e(M,t)$.

\begin{theorem}\label{Th4.1}
Suppose that the coefficient $\alpha\geq \alpha_0>0$ and the initial
value  $u_0(x)$  satisfies
\begin{equation}\label{4.041}
\left\|u_0\right\|_{L^2(0,1)}\leq M_0
\end{equation}
for some $M_0>0$. Then for $t\in [T_1,T_2]$,
\begin{equation}\label{4.042}
|e(M,t)|<\left(\sqrt{2}+\dfrac{1}{4M\pi^2\alpha_0
T_1}\right)M_0e^{-\alpha_0 M^2\pi^2 T_1}.
\end{equation}
\end{theorem}
\begin{proof}
According to \dref{3.2} and \dref{4.041},
\begin{equation*}
\sum_{n=0}^\infty |C_n|^2\leq 2\|u_0\|^2 \leq 2M_0^2.
\end{equation*}
Introduce
\begin{equation}\label{4.04}
F(\alpha,M,t)=\int_M^\infty e^{-\alpha x^2\pi^2 t} \;dx,
\end{equation}
and three regions in  $(x,\,y)$ plane:
$$D=\left\{(x,y)\,\bigg|\,x\geq M,\; y\geq M\right\},\;D_1=\left\{(x,y)\,\bigg|\,x\geq y\geq M\right\},\;D_2=\left\{(x,y)\,\bigg|\,y>x\geq M\right\}.$$ It is obvious that
$D=D_1\cup D_2$ and $D_1\cap D_2=\emptyset$, and
$$
\begin{array}{ll}
F^2(\alpha,M,t)&=\disp \int_M^\infty e^{-\alpha x^2\pi^2 t} \;
dx\cdot \int_M^\infty e^{-\alpha y^2\pi^2 t}\;dy=\disp\iint_D
e^{-\alpha (x^2+y^2)\pi^2 t}\;dx dy\crr &=\disp \iint_{D_1}
e^{-\alpha (x^2+y^2)\pi^2 t} \; dx dy+\iint\limits_{D_2} e^{-\alpha
(x^2+y^2)\pi^2 t} \; dx dy\crr &\disp \triangleq I_1+I_2.
\end{array}
$$
By symmetry of the integration domains $D_1$ and $D_2$ with respect
to $x$ and $y$,  $I_1=I_2$. To compute  $I_1$, we use a double
integral in polar coordinates to convert it to iterated integrals.
The region $D_1$ under polar coordinates becomes $\widetilde{D}_1$:
$$
\widetilde{D}_1=\left\{(\rho,\theta)\,\bigg|\,\rho\geq
\dfrac{M}{\sin\theta},\;0<\theta\leq \dfrac{\pi}{4}\right\}.
$$
Then we can rewrite the double integral $I_1$ as an iterated
integral in polar coordinates:
\begin{equation}\label{4.09}
\begin{array}{ll}
I_1&=\disp \iint_{D_1} e^{-\alpha (x^2+y^2)\pi^2 t} \; dx dy
=\iint\limits_{\widetilde{D}_1} e^{-\alpha \rho^2 \pi^2 t}\rho \;
d\rho d\theta\crr &=\disp \int_0^\frac{\pi}{4} \; d\theta
\int_{\frac{M}{\sin\theta}}^\infty e^{-\alpha \rho^2 \pi^2 t}\rho \;
d\rho=\dfrac{1}{2\alpha\pi^2 t}\int_0^\frac{\pi}{4} e^{-\frac{\alpha
M^2\pi^2 t}{\sin^2 \theta}} \; d\theta.
\end{array}
\end{equation}
The variable substitution $u=\cot^2 \theta$ in \dref{4.09} yields
\begin{eqnarray}\label{4.10}
\nonumber I_1&=&\dfrac{1}{4\alpha\pi^2 t}\int_1^{+\infty} e^{-\alpha
M^2\pi^2 t
(1+u)}\dfrac{1}{(1+u)\sqrt{u}} \; du\\
\nonumber&<&\dfrac{1}{4\alpha\pi^2 t}e^{-\alpha M^2\pi^2
t}\int_{1}^{+\infty} e^{-\alpha M^2\pi^2 t u}u^{-\frac{3}{2}} \; du\\
&=&\dfrac{M}{4\pi\sqrt{\alpha t}}e^{-\alpha M^2\pi^2 t}\cdot\Gamma
\left(-\frac{1}{2},\;\alpha M^2\pi^2 t\right),
\end{eqnarray}
where
\begin{equation}\label{4.11}
\Gamma (a,x)=\int_x^\infty t^{a-1}e^{-t} \; dt,
\end{equation}
is the upper incomplete gamma function (\cite[Section
6.5]{Abramowitz}).  It  is known that \cite[p.263]{Abramowitz}
\begin{equation}\label{4.12}
\Gamma
(a,x)=e^{-x}x^a\left(\dfrac{1}{x+}\dfrac{1-a}{1+}\dfrac{1}{x+}\dfrac{2-a}{1+}\dfrac{2}{x+}\cdots\right),\;x>0,\;
|a|<\infty,
\end{equation}
from which we have
\begin{equation}\label{4.13}
\Gamma (a,x)<e^{-x}x^{a-1},\; x>0,\; a<0.
\end{equation}
This together with  \dref{4.10} gives
\begin{equation}\label{4.14}
I_1<\dfrac{1}{4M^2\pi^4\alpha^2 t^2}e^{-2\alpha M^2\pi^2 t}.
\end{equation}
Therefore,
\begin{equation}\label{4.05}
F(\alpha,M,t)=\sqrt{2I_1}<\dfrac{1}{\sqrt{2}M\pi^2\alpha
t}e^{-\alpha M^2\pi^2 t}.
\end{equation}
We now turn to the estimation of $|e(M,t)|$. It is computed that
\begin{equation}\label{4.06}
\begin{array}{ll}
\nonumber|e(M,t)|&\leq \disp\sum_{n=M}^\infty |C_n|e^{-\alpha
n^2\pi^2 t}\leq \left(\disp\sum_{n=M}^\infty
|C_n|^2\right)^\frac{1}{2}\left(\disp\sum_{n=M}^\infty e^{-2\alpha
n^2\pi^2 t}\right)^\frac{1}{2}\crr &< \disp
\sqrt{2}M_0\left[e^{-2\alpha M^2\pi^2 t}+\sum_{n=M+1}^\infty
\int_{n-1}^n e^{-2\alpha x^2\pi^2 t} \; dx\right]^\frac{1}{2}\crr
&=\sqrt{2}M_0\cdot\left[e^{-2\alpha M^2\pi^2 t}+
F(2\alpha,M,t)\right]^\frac{1}{2}\crr  &<\disp
\left(\sqrt{2}+\dfrac{1}{4M\pi^2\alpha t}\right)M_0e^{-\alpha
M^2\pi^2 t}.
\end{array}
\end{equation}
Since $\alpha\geq\alpha_0$, we finally obtain
\begin{equation}\label{4.07}
|e(M,t)|<\left(\sqrt{2}+\dfrac{1}{4M\pi^2\alpha_0
T_1}\right)M_0e^{-\alpha_0 M^2\pi^2 T_1},\;\forall\; t\in [T_1,T_2].
\end{equation}
This completes the proof of the theorem.
\end{proof}
\begin{remark}\label{Re4.1.1} It is seen from \dref{4.042} that
if $\alpha_0 M^2\pi^2 T_1$ is sufficiently large, then it indeed has
\begin{equation}\label{4.08} y(t)\approx\sum_{n=0}^{M-1} C_n
e^{-\lambda_n t},\; \forall\;  t\in [T_1,T_2],
\end{equation}
with the truncation error $e(M,t)$ estimation \dref{4.042}.
\end{remark}

\subsection{Matrix pencil method}

The matrix pencil method was first presented by  Hua and Sarkar in
\cite{Hua1,Hua2} for estimating signal parameters from a noisy
exponential  sequence. This  method has been proved to be quite
useful because of its computational efficiency and   low sensitivity
to the noise.

Suppose that the observed system response can be described by
\begin{equation}\label{4.2.1}
y(t)=x(t)+n(t)=\disp\sum_{i=1}^M R_i \exp(s_i t)+n(t),\; \forall\;
t\in [0,T].
\end{equation}
where $n(t)$ is the noise, $x(t)$ is the system response, $y(t)$ is
the noise contaminated observation, $M$ is the number of exponential
components, and $T$ is the maximal observation coverage time.

Let $T_s$ be the sampling period. The discrete form of \dref{4.2.1}
can be expressed as follows:
\begin{equation}\label{4.2.2}
y(kT_s)=x(kT_s)+n(kT_s)=\disp\sum_{i=1}^M R_i z_i^k+n(kT_s),\;
k=0,1,\dots,N-1,
\end{equation}
where $z_i=\exp(s_i T_s)$ are  the poles of response signal, and $N$
is the number of sample points which should be large enough.
Generally, all the number of exponential components $M$, the
amplitudes $R_i$, and the poles $z_i$ can be unknown. In what
follows, we show how to estimate these numbers simultaneously from
the observation $\left\{y(kT_s)\right\}_{k=0}^{N-1}$ by virtue of
the matrix pencil method.

Let $x_k=x(kT_s)$ and $y_k=y(kT_s)$, and define
\begin{equation}\label{4.2.3}
{\bf{x}_t}=[x_t,x_{t+1},\dots,x_{N-L+t-1}]^\top,\;
{\bf{y}_t}=[y_t,y_{t+1},\dots,y_{N-L+t-1}]^\top,\;t=0,1,2,\dots,L
\end{equation}
and
\begin{equation}\label{4.2.4}
\begin{array}{l}
\disp
X_0=[{\bf{x}_{L-1}},{\bf{x}_{L-2}},\dots,{\bf{x}_0}],\;\;Y_0=[{\bf{y}_{L-1}},{\bf{y}_{L-2}},\dots,{\bf{y}_0}],\crr\disp
X_1=[{\bf{x}_{L}},{\bf{x}_{L-1}},\dots,{\bf{x}_1}],\;\;Y_1=[{\bf{y}_{L}},{\bf{y}_{L-1}},\dots,{\bf{y}_1}],\crr\disp
X=[{\bf{x}_{0}},{\bf{x}_{1}},\dots,{\bf{x}_L}],\;\;Y=[{\bf{y}_{0}},{\bf{y}_{1}},\dots,{\bf{y}_L}],
\end{array}
\end{equation}
where the superscript ``$\top$" denotes the transpose, and  $L$ is
called the pencil parameter. It has been pointed out that the best
choices for $L$ are $N/3$ and $2N/3$, and all values satisfying
$N/3\leq L \leq 2N/3$ appear to be good choices in general
\cite{Hua3}. In this paper, the pencil parameter $L$ is always
chosen to be $N/3$ or $\lfloor N/3 \rfloor+1$ when $N/3$ is not an
integer. Here and in the sequel,  $\lfloor \cdot\rfloor$ is as usual
the floor function and $\lfloor x \rfloor$ denotes the integer part
of the number $x$.

Suppose that the singular value decomposition (SVD) of $Y$ is
$Y=U\Sigma V^\top,$ where $U$ and $V$ are $(N-L)\times (N-L)$ and
$(L+1)\times (L+1)$ orthogonal matrices, respectively, $\Sigma$ is
an $(N-L)\times (L+1)$ diagonal matrix with entries $\{\sigma_i\}$
in main diagonal to be the singular values of $Y$.

\subsubsection{The estimation of $M$}

In case of noiseless observation, {\it i.e.} $n(kT_s)=0$ in
\dref{4.2.2}, $M$ is equal to the number of nonzero singular values
of the matrix $X$ defined in \dref{4.2.4}, or equivalently the rank
of $X$, that is, $M=\rank(X).$

In case of the noise contaminated observation,  however, the
elements that are originally zeros in main diagonal of $\Sigma$
might not be zeros anymore due to influence of noise. Nevertheless,
the values of these elements will be very small as long as the noise
is very weak in comparison to the signal (see, e.g., \cite{Mirsky}).
Thus, an effective practical method for estimating the number $M$ is
first to choose the maximal singular value $\sigma_{max}$ of $Y$ and
assign a threshold $\varepsilon$ for the singular values, e.g.,
$\varepsilon=10^{-10}$, and then treat any small singular value
$\sigma_i$ which satisfies $\sigma_i/\sigma_{max}<\varepsilon$ to be
zero. Therefore, $M$ can be estimated by
\begin{equation}\label{4.2.8}
M=\#\bigg\{\sigma_i\Big|\; \sigma_i\; \mbox{is the singular value
of}\; Y\;\mbox{which satisfies}\;\dfrac{\sigma_i}{\sigma_{max}}\geq
\varepsilon\bigg\}.
\end{equation}
where $\# S$, here and in the sequel, denotes the number of elements
in the set $S$.

\subsubsection{Estimation of  poles $\left\{z_i\right\}_{i=1}^M$}

In case of noiseless observation, it has been proved in
\cite[Theorem 2.1]{Hua3} that the poles $\left\{z_i\right\}_{i=1}^M$
in \dref{4.2.2} are the $M$ eigenvalues of the matrix $X_0^\dag X_1$
when $M\leq L\leq N-M$, here and in the sequel the superscript
``\dag" denotes the Moore-Penrose inverse or pseudoinverse. Since
$X_0^\dag X_1$ has rank $M\leq L$, there are also $L-M$ zero
eigenvalues for the matrix product.

In case of the noise contaminated observation, suppose that the SVD
of $Y_0$ is $Y_0=U_0\Sigma_0 V_0^\top$, and the rank-$M$ truncated
pseudoinverse $Y_{0,M}^\dag$ is defined as
\begin{equation}\label{4.2.10}
Y_{0,M}^\dag =\sum_{i=1}^M \dfrac{1}{\sigma_i}v_i u_i^\ast
=V_{0,M}A^{-1} U_{0,M}^*,
\end{equation}
where $\left\{\sigma_i\right\}_{i=1}^M$ are the $M$ largest singular
values of $Y_0$; $v_i$'s and $u_i$'s are the corresponding singular
vectors, and
\begin{equation}\label{4.2.10.2}
V_{0,M}=\{v_1,v_2,\dots,v_M\},\; U_{0,M}=\{u_1,u_2,\dots,u_M\},\;
A=\diag\{\sigma_1,\sigma_2,\dots,\sigma_M\}.
\end{equation}
The superscript ``$*$'' in \dref{4.2.10} denotes the conjugate
transpose.

It is shown in \cite{Hua3} that the estimates of the poles
$\left\{z_i\right\}_{i=1}^M$ can be realized by computing the $M$
nonzero eigenvalues of $Y_{0,M}^\dag Y_1$, or equivalently, the
eigenvalues of the $M\times M$ matrix
\begin{equation}\label{4.2.11}
Z_E=A^{-1}U_{0,M}^\ast Y_1V_{0,M}.
\end{equation}
Then the $\left\{s_i\right\}_{i=1}^M$ in \dref{4.2.1} can be
obtained by
\begin{equation}\label{4.2.12}
s_i=\dfrac{\ln z_i}{T_s},\; i=1,2,\dots,M.
\end{equation}

\begin{remark}\label{Re4.33.1} It is easily seen that
the matrix pencil method contains truncated singular value
decomposition (TSVD) (see, e.g.,  \cite{Hansen})
 as a regularization method
to estimate $M$ and $\{z_i\}_{i=1}^M$.
\end{remark}

\subsubsection{Estimation of amplitudes $\left\{R_i\right\}_{i=1}^M$}

Having estimated the number $M$ of the  exponential components,  and
all the poles $\left\{z_i\right\}_{i=1}^M$, the amplitudes $R_i$ can
be estimated by solving the following linear least squares problem
\begin{equation}\label{4.2.13}
\left\{R_i\right\}_{i=1}^M=\argmin \sum_{k=0}^{N-1} \bigg[
y_k-\sum_{i=1}^M R_i z_i^k\bigg]^2.
\end{equation}

\subsection{Identification algorithm}

Suppose that $0<T_1<T_2<T_3$ are three arbitrary positive numbers,
and the control function $f(t)$ is chosen  as in  \dref{3.8} and the
corresponding observation data is $\{y(t)=u(0,t;f,u_0)|\; t\in [T_1,
T_3]\}$. In this section, we formulate the identification for the
coefficient and initial value in several steps.

{\it  Step 1: Estimate several eigenvalues of system operator $\A$
from the observation without control by the matrix pencil method.}

Specifically, let $T_1=t_0<t_1<\cdots <t_{N_1}=T_2$ be the uniform
grids of $[T_1,T_2]$ with the sampling period
$T_s=\frac{T_2-T_1}{N_1}$, and the measured values at sample points
are
\begin{equation}\label{4.21}
y_i=y(t_i)=\disp\sum_{n=0}^\infty C_n e^{-\lambda_n
t_i}=\disp\sum_{k=0}^{K-1} \left(C_{n_k}e^{-\lambda_{n_k}
T_1}\right) e^{-(\lambda_{n_k} T_s) i},\; i=0,1,\dots,N_1-1,
\end{equation}
where $K=\# \K$  with $\K$ being  defined by \dref{3.3} and the
series $\left\{C_{n_k}\right\}_{k=0}^{K-1}$ consists of all the
nonzero elements in the series $\left\{C_n\right\}_{n\in \N_0}$ by
removing all the zero ones. Then the number $M$ of the estimable
eigenvalues and the approximate eigenvalues
$\left\{\widetilde{\lambda}_{n_k}\right\}_{k=0}^{M-1}$ can be
obtained by virtue of the matrix pencil method following the process
introduced in sections 3.2.1 and 3.2.2.

\begin{remark}\label{Re4.3.1}
As stated in Theorem \ref{Th3.1}, it is unknown whether the initial
value $u_0(x)$ is orthogonal to some of the eigenvectors
$\{\phi_n\}_{n\in \N_0}$. In case that $\langle u_0,\phi_n\rangle
=0$ for some $n\in \N_0$, then $C_n=0$ and the observation has
nothing to do with the term $C_n e^{-\lambda_n t_i}$. It is
noteworthy that the
$\left\{\widetilde{\lambda}_{n_k}\right\}_{k=0}^{M-1}$ recovered in
{\it Step 1} are the approximations of some eigenvalues of system
operator $\A$, but may not be the first $M$ eigenvalues, {\it i.e.}
the relationships $\widetilde{\lambda}_{n_k}\approx\lambda_k\;
(=\alpha k^2 \pi^2)$ are not always true. In fact, it is true only
when $n_k=k$ or $\langle u_0,\phi_k\rangle \neq 0$ for
$k=0,1,\dots,M-1$, which is the case mentioned in \cite{Suzuki1},
where such an initial value is said to be generic and in this case
the {\it Steps 3} and {\it 4} below are not necessary anymore. In
other words, when $\langle u_0,\phi_k\rangle =0$ for some $k$, we
can not always recover $\alpha$ from
$\left\{\widetilde{\lambda}_{n_k}\right\}$ directly.
\end{remark}

{\it  Step 2: Estimate the coefficients
$\left\{\widetilde{C}_{n_k}\right\}_{k=0}^{M-1}$ that are
corresponding to
$\left\{\widetilde{\lambda}_{n_k}\right\}_{k=0}^{M-1}$ from
\dref{4.21} by solving the linear least square problem following
\begin{equation}\label{4.23}
\left\{\widetilde{C}_{n_k}\right\}_{k=0}^{M-1}=\argmin\sum_{i=0}^{N_1-1}
\left[y_i-\sum_{k=0}^{M-1}\widetilde{C}_{n_k}
e^{-\widetilde{\lambda}_{n_k} t_i}\right]^2.
\end{equation}}

\begin{remark}\label{Re4.3.2}
After obtaining
$\left\{\left(\widetilde{C}_{n_k},\widetilde{\lambda}_{n_k}\right)\right\}_{k=0}^{M-1}$,
the  control free part of the observation $u(0,t;0,u_0)$ can be
estimated as
\begin{equation}\label{4.24}
u(0,t;0,u_0)\approx \sum_{k=0}^{M-1}\widetilde{C}_{n_k}
e^{-\widetilde{\lambda}_{n_k} t},\; t>0.
\end{equation}
\end{remark}

{\it  Step 3:  Estimate an  approximation of $\alpha$ by obtaining
the first several eigenvalues of the operator $\A$ through the
observation data $\{y(t)|\; t\in [T_2,\;T_3]\}$ by virtue of the
matrix pencil method.}

Similar to {\it Step 1}, let $T_2=t_0<t_1<\cdots <t_{N_2}=T_3$ be
the uniform grids of $[T_2,T_3]$ with the sampling period
$T_s^\prime=\frac{T_3-T_2}{N_2}$, and the control is chosen to be
$f(t)=1$ for $t\in [T_2,T_3]$. Then  from \dref{4.24} we obtain
\begin{equation}\label{4.25}
\begin{array}{ll}
y(t_i)&=\disp u(0,t_i;0,u_0)+u(0,t_i;f,0)\crr &=\disp
\sum_{n=0}^\infty C_ne^{-\lambda_n t_i}-\int_0^{t_i}
G(t_i-s,0,0)f(s)ds\crr &\disp \approx\sum_{k=0}^{M-1}
\widetilde{C}_{n_k} e^{-\widetilde{\lambda}_{n_k} t_i}-\dfrac{1}{3
\alpha}-(t_i-T_2)+\disp\sum_{n=1}^\infty\dfrac{2}{\lambda_n}e^{-\lambda_n
(t_i-T_2)}\crr &=\disp \sum_{k=0}^{M-1} \widetilde{C}_{n_k}
e^{-\widetilde{\lambda}_{n_k} t_i}-\dfrac{1}{3 \alpha}-T_s^\prime
i+\sum_{n=1}^\infty\dfrac{2}{\lambda_n}e^{-\lambda_n T_s^\prime i}.
\end{array}
\end{equation}
Let
\begin{equation}\label{4.26.1}
y_i^\prime=y(t_i)-\sum_{k=0}^{M-1} \widetilde{C}_{n_k}
e^{-\widetilde{\lambda}_{n_k} t_i}+T_s^\prime i,\;
i=0,1,\dots,N_2-1,
\end{equation}
and
\begin{equation}\label{4.26}
C_0^\prime=-\dfrac{1}{3\alpha},\;\; \lambda_0^\prime=0,\;
C_n^\prime=\dfrac{2}{\lambda_n},\;\; \lambda_n^\prime=\lambda_n
T_s^\prime,\; n\in\N.
\end{equation}
Then \dref{4.25} becomes
\begin{equation}\label{4.27}
y_i^\prime\approx\sum_{n=0}^\infty C_n^\prime e^{-\lambda_n^\prime
i},\; i=0,1,\dots,N_2-1.
\end{equation}

Next, we estimate $\left\{\left(C_n^\prime,
\lambda_n^\prime\right)\right\}_{n=0}^{M^\prime-1}$ from \dref{4.27}
by repeating the processes  in {\it Steps 1} and {\it 2}. Then
$\alpha$ can be obtained from \dref{2.2} and \dref{4.26}.
\begin{remark}\label{Re4.3.3}
The estimation process for $\left\{\left(C_n^\prime,
\lambda_n^\prime\right)\right\}_{n=0}^{M^\prime-1}$ from \dref{4.27}
is slightly different in {\it Step 1} since none of the
$\left\{C_n^\prime\right\}_{n=0}^{M^\prime-1}$ is zero although they
are also unknown. Hence, we can recover $\alpha$ from one of the
following relations:
\begin{equation}\label{4.28}
\lambda_n^\prime=\alpha n^2\pi^2T_s^\prime,\;
n=1,2,\dots,M^\prime-1,
\end{equation}
and
\begin{equation}\label{4.29}
C_n^\prime=\dfrac{2}{\alpha n^2\pi^2},\; n=1,2,\dots,M^\prime-1.
\end{equation}
However, the $\alpha$ obtained from \dref{4.28} may be different
from that obtained from \dref{4.29} since both
$\left\{C_n^\prime\right\}$ and $\left\{\lambda_n^\prime\right\}$
are estimated values rather than exact ones. In  simulations, the
pairs $(C_n^\prime,\lambda_n^\prime)$ that satisfy
\begin{equation}\label{4.30}
C_n^\prime \lambda_n^\prime\approx 2T_s^\prime, \;
n=1,2,\dots,M^\prime-1.
\end{equation}
seem to be more credible to estimate $\alpha$. Actually, the
estimated coefficient here is only for identification
$\left\{n_k\right\}_{k=0}^{M-1}$ from
$\left\{\widetilde{\lambda}_{n_k}\right\}_{k=0}^{M-1}$ which is
shown in succeeding {\it Step 4}. Finally, we emphasize that the
identification of $\alpha$ does not depend on the sampling period
but the special structure of eigenvalues \dref{4.28}. If there is no
such structure for eigenvalues, our idea of transforming the
identification of $\alpha$ to be a zero initial value problem can
simplify the problem.
\end{remark}

{\it  Step 4:  Estimate $\alpha$ from
$\left\{\widetilde{\lambda}_{n_k}\right\}_{k=0}^{M-1}$ and
reconstruct the initial state $u_0(x)$}.

To be specific, after estimating
$\left\{\left(\widetilde{C}_{n_k},\widetilde{\lambda}_{n_k}\right)\right\}_{k=0}^{M-1}$
in {\it Steps 1, 2}, and recovering an approximation of $\alpha$ in
{\it Step 3}, we can now determine the series
$\K_M=\left\{n_k\right\}_{k=0}^{M-1}$ by
\begin{equation}\label{4.31}
n_k=\left\lfloor \sqrt{\dfrac{\widetilde{\lambda}_{n_k}}{\alpha
\pi^2}}\;\right\rceil,\; k=0,1,\dots,M-1,
\end{equation}
where $\lfloor x \rceil$ denotes the integer nearest to $x$. Then,
$\alpha$ can be estimated by
\begin{equation}\label{4.1.37}
\alpha_k=\dfrac{\widetilde{\lambda}_{n_k}}{n_k^2\pi^2} \mbox{  for }
n_k\neq0, \;  k=0,1,\dots,M-1.
\end{equation}
An error analysis between the estimated coefficient $\alpha_k$ and
the real value is discussed in section 5.

Now we turn to initial value. It is clear from \dref{2.11} that
\begin{equation}\label{4.33}
y(t)=\disp\sum_{n=0}^\infty A_n(0)e^{-\lambda_n t}
=\disp\sum_{n=0}^{\widetilde{M}-1} A_n(0)e^{-\alpha n^2\pi^2
t}+e(\widetilde{M},t),\; \forall\; t\in   [T_0,T_2),
\end{equation}
where $T_0\in (0,T_2)$. It follows from Theorem \ref{Th4.1} that we
can choose proper $\widetilde{M}$ and $T_0$ such that
$\left|e(\widetilde{M},t)\right|$ is sufficiently small. Suppose
that only observation at the sample points
$T_0=t_0<t_1<\dots<t_{N}=T_2$ are available. Then the coefficients
$\left\{A_n(0)\right\}$ can be estimated by solving the following
problem properly
\begin{equation}\label{4.35}
\sum_{i=0}^{N-1} \left[y(t_i)-\sum_{n=0}^{\widetilde{M}-1}
A_n(0)e^{-\alpha n^2\pi^2 t_i}\right]^2,
\end{equation}
or equivalently, finding the least squares solution of the matrix
equation
\begin{equation}\label{4.35.2}
CA=b,
\end{equation}
where $C$ is an $N\times \widetilde{M}$ matrix with the $(i,j)$
element
\begin{equation}\label{4.35.3}
C(i,j)=e^{-\alpha (j-1)^2\pi^2 t_{i-1}},
\end{equation}
and
\begin{equation}\label{4.35.4}
A=[A_0(0),A_1(0),\cdots,A_{\widetilde{M}-1}(0)]^\top,\;
b=[y(t_0),y(t_1),\cdots,y(t_{N-1})]^\top.
\end{equation}
Since the reconstruction of the initial value is known to be
ill-posed, which results in  the resulting matrix equation
\dref{4.35.2} to be ill-posed as well. In order to obtain stable
results, some regularization method is required. Here we use the
TSVD \cite{Hansen} to solve the matrix equation \dref{4.35.2}.

Suppose that the SVD of matrix $C$ is
\begin{equation}\label{4.35.5}
C=U_C\Sigma_C V_C^\top,
\end{equation}
where $U_C=[u_1^\prime,u_2^\prime,\cdots,u_N^\prime]$ and
$V_C=[v_1^\prime,v_2^\prime,\cdots,v_{\widetilde{M}}^\prime]$ are
orthonormal matrices with column vectors named left and right
singular vectors, respectively. $\Sigma_C=\diag
(\sigma_1,\sigma_2,\cdots)$ is a diagonal matrix with non-negative
diagonal elements being the singular values of $C$. In the TSVD
method, the matrix $C$ is replaced by its rank-$k$ approximation,
and the regularized solution is given by
\begin{equation}\label{4.35.6}
A_{reg}=\sum_{i=1}^k\dfrac{u_i^{\prime\top}b}{\sigma_i}v_i^\prime.
\end{equation}
where $k\leq \rank(C)$ is the regularization parameter. In this
paper, we use the generalized cross-validation (GCV) criterion
\cite{Golub} to determine the regularization parameter. The GCV
criterion determines the optimal regularization parameter $k$ by
minimizing the following GCV function:
\begin{equation}\label{4.35.7}
G(k)=\dfrac{\|CA_{reg}-b\|^2}{(trace(I_N-CC^I))^2},
\end{equation}
where $C^I$ is the matrix which produces the regularized solution
after being multiplied with the right-hand side $b$, i.e.
$A_{reg}=C^Ib$.

Having obtained the regularized solution $A_{reg}$, then the initial
value can be estimated by the asymptotic Fourier series expansion:
\begin{equation}\label{4.36}
u_0(x)\approx \sum_{n=0}^{\widetilde{M}-1} A_n(0) \cos n \pi x.
\end{equation}

\begin{remark}\label{Re4.3.4}
It is obvious that the reconstructed  initial value
$\widetilde{u}_0(x)$ by \dref{4.36} is an approximated Fourier
series expansion of $u_0(x)$ with the first $\widetilde{M}$ terms.
In fact, since  $\alpha$ has been estimated, there are various
methods for the initial state reconstruction, see, e.g.,
\cite{Ramdani,Xu} and the references therein. Compared with those
methods, the method here is more direct and simple.
\end{remark}

\section{Error analysis}

Noise sensitivity  of the matrix pencil method for estimating finite
signal parameters from a noisy exponential sequence is analyzed
 in \cite{Hua3}.  But our case is different in  two aspects. First,
the number of unknown parameters in  the infinite spectral
estimation is not finite. Second, the perturbation, that is, the
remainder term $e(M,t)$ in \dref{4.03},  is not random.  In this
section, we establish an error analysis by applying the matrix
pencil method to the infinite spectral estimation problem:
\begin{equation}\label{5.1.1}
y(t)=\sum_{n=0}^\infty C_n e^{-\lambda_n t},\;  \forall\; t\in
[T_1,T_2].
\end{equation}
We may suppose without loss of generality that $C_n\neq 0$ for any
$n\in\N_0$. In fact, we are only concerned with the first $M$
nonzero terms in series \dref{5.1.1} which is written in a clear way
as
\begin{equation}\label{5.1.2}
y(t)=x(t)+e(M,t)=\sum_{n=0}^{M-1} C_n e^{-\lambda_n
t}+\sum_{n=M}^\infty C_n e^{-\lambda_n t},\; \forall\; t\in
[T_1,T_2],
\end{equation}
where $M$ is defined as \dref{4.2.8}. Let $T_1=t_0<t_1<\cdots
<t_{N-1}=T_2$ be the points on a uniform grid of $[T_1,T_2]$ with
the sampling period $T_s=\frac{T_2-T_1}{N-1}$, and hence the
observation data at sample points, $t_i$, are
\begin{equation}\label{5.3.5}
\begin{array}{ll}
y_i=y(t_i)&=\disp \sum_{n=0}^\infty \left(C_{n}e^{-\lambda_{n}
T_1}\right) e^{-(\lambda_{n} T_s) i}= \sum_{n=0}^{M-1}
\left(C_{n}e^{-\lambda_{n} T_1}\right) e^{-(\lambda_{n} T_s)
i}+e(M,t_i)\crr  &\disp \triangleq\sum_{n=0}^{M-1}
\left(C_{n}e^{-\lambda_{n} T_1}\right) z_n^i+e(M,t_i),
\end{array}
\end{equation}
where $z_n=e^{-\lambda_n T_s}$. By Theorem \ref{Th4.1}, it follows
that
\begin{equation}\label{5.1.4}
\begin{array}{ll}
|y_i-x_i|=|e(M,t_i)|& \disp
<\left(\sqrt{2}+\dfrac{1}{4M\pi^2\alpha_0 t_i}\right)M_0e^{-\alpha
M^2\pi^2 t_i}\crr &\disp
\leq\left(\sqrt{2}+\dfrac{1}{4M\pi^2\alpha_0
T_1}\right)M_0e^{-\alpha_0 M^2\pi^2 T_1}e^{-\alpha_0 M^2\pi^2 T_s
i},
\end{array}
\end{equation}
where $x_i=\sum_{n=0}^{M-1} \left(C_{n}e^{-\lambda_{n} T_1}\right)
z_n^i$.  Define the matrices $X_0,\;Y_0,\;X_1,\;Y_1$ as
\dref{4.2.3}-\dref{4.2.4}. Theorem \ref{Th5.1.1} below gives the bounds of
$\|Y_0-X_0\|_F$ and $\|Y_1-X_1\|_F$, where $\|\cdot\|_F$ denotes the
matrix Frobenius norm.

\begin{theorem}\label{Th5.1.1}
Let the number of sample points $N>9$ and
\begin{equation}\label{5.2.1}
\theta=2\alpha_0 M^2\pi^2T_s.
\end{equation}
Then
\begin{equation}\label{5.1.5}
\|Y_0-X_0\|_F<\left(\sqrt{2}+\dfrac{1}{4M\pi^2\alpha_0
T_1}\right)M_0e^{-\alpha_0 M^2\pi^2
T_1}\sqrt{M_{\theta,L}+\left(1+\dfrac{1}{\theta}\right)^2},
\end{equation}
and
\begin{equation}\label{5.1.6}
\|Y_1-X_1\|_F<\left(\sqrt{2}+\dfrac{1}{4M\pi^2\alpha_0
T_1}\right)M_0e^{-\alpha_0 M^2\pi^2
T_1}\sqrt{M_{\theta,L+1}+\dfrac{1}{\theta}\left(1+\dfrac{1}{\theta}\right)e^{-\theta}},
\end{equation}
where
\begin{equation}\label{5.2.11}
M_{\theta,L}=\left\{
\begin{array}{ll}
e^{-\theta},& \theta\geq 1,\\
\dfrac{2}{\theta}e^{-1},& \dfrac{1}{L-1}<\theta<1,\\
(L-1)e^{-(L-1)\theta},& 0<\theta\leq \dfrac{1}{L-1},
\end{array}\right.
\end{equation}
\end{theorem}
\begin{proof}
Since  both  matrices $X_0$ and $Y_0$ admit the Hankel structure, it
is easy to deduce from the definition of Frobenius norm that
$$
\begin{array}{ll}
\|Y_0-X_0\|_F^2&=\disp
\sum_{i=0}^{L-1}(i+1)\left|y_i-x_i\right|^2+\sum_{i=1}^L
i\left|y_{N-1-i}-x_{N-1-i}\right|^2
+L\sum_{i=L}^{N-L-2}\left|y_i-x_i\right|^2\crr &\leq\disp
\left(\sqrt{2}+\dfrac{1}{4M\pi^2\alpha_0
T_1}\right)^2M_0^2e^{-2\alpha_0 M^2\pi^2
T_1}\bigg[\sum_{i=0}^{L-1}ie^{-\theta i}+\sum_{i=0}^{L-1}e^{-\theta
i}\crr  &\hspace{0.3cm}+\disp \sum_{i=1}^L ie^{-\theta (N-1-i)}
+L\sum_{i=L}^{N-L-2}e^{-\theta i} \bigg]\crr &\disp \triangleq
\left(\sqrt{2}+\dfrac{1}{4M\pi^2\alpha_0
T_1}\right)^2M_0^2e^{-2\alpha_0 M^2\pi^2
T_1}\left[S_1+S_2+S_3+S_4\right].
\end{array}
$$
To  estimate $S_1$, we introduce
\begin{equation}\label{5.1.7}
f(x)=xe^{-\theta x},\; x\geq 0,
\end{equation}
which satisfies
$$f^\prime (x)=(1-\theta x)e^{-\theta x},\; x\geq 0.$$
There are three different cases according to the values of $\theta$.

{\it  Case 1:  $\theta\geq 1$.}  In this case,  $f^\prime (x)\leq 0$
for  $x\geq 1$. Hence
\begin{equation}\label{5.1.8} ie^{-\theta
i}\leq \int_{i-1}^i xe^{-\theta x}dx,\; i=2,3,\dots, L-1.
\end{equation}
Therefore,
$$
S_1=e^{-\theta}+\sum_{i=2}^{L-1}ie^{-\theta i} \leq
e^{-\theta}+\sum_{i=2}^{L-1}\int_{i-1}^i xe^{-\theta x}dx
=e^{-\theta}+\int_{1}^{L-1} xe^{-\theta x}dx.
$$

{\it  Case 2:  $\frac{1}{L-1}<\theta<1$.}  In this case,   $f^\prime
(x)\geq 0$ for $0\leq x\leq \dfrac{1}{\theta}$, and $f^\prime (x)<0$
for $x> \dfrac{1}{\theta}$, which imply
$$f(x)\leq f({\theta}^{-1})=\dfrac{1}{\theta}e^{-1},\; x\geq
0,
$$
and
\begin{equation}\label{5.2.8} ie^{-\theta i}\leq \int_i^{i+1}
xe^{-\theta x}dx,\; i=1,\dots,\left\lfloor
\dfrac{1}{\theta}\right\rfloor -1.
\end{equation}
\begin{equation}\label{5.2.9}
ie^{-\theta i}\leq \int_{i-1}^i xe^{-\theta x}dx,\; i=\left\lfloor
\dfrac{1}{\theta}\right\rfloor +2,\dots,L-1.
\end{equation}
Thus
\begin{eqnarray*}
S_1&=&\sum_{i=1}^{L-1}ie^{-\theta i}
=\sum_{i=1}^{\lfloor\frac{1}{\theta}\rfloor-1}ie^{-\theta i}
+\left\lfloor\dfrac{1}{\theta}\right\rfloor e^{-\theta
\lfloor\frac{1}{\theta}\rfloor}+\left(\left\lfloor\dfrac{1}{\theta}\right\rfloor+1\right)e^{-\theta
(\lfloor\frac{1}{\theta}\rfloor+1)}
+\sum_{i=\lfloor\frac{1}{\theta}\rfloor+2}^{L-1}ie^{-\theta i}\\
&\leq&\sum_{i=1}^{\lfloor\frac{1}{\theta}\rfloor-1}\int_i^{i+1}xe^{-\theta
x}dx
+f\left(\left\lfloor\dfrac{1}{\theta}\right\rfloor\right)+f\left(\left\lfloor\dfrac{1}{\theta}\right\rfloor+1\right)
+\sum_{i=\lfloor\frac{1}{\theta}\rfloor+2}^{L-1}\int_{i-1}^ixe^{-\theta x}dx\\
&\leq&\int_1^{\lfloor\frac{1}{\theta}\rfloor}xe^{-\theta x}dx
+2f\left(\dfrac{1}{\theta}\right)+\int_{\lfloor\frac{1}{\theta}\rfloor+1}^{L-1} xe^{-\theta x}dx\\
&\leq&\dfrac{2}{\theta}e^{-1}+\int_{1}^{L-1} xe^{-\theta x}dx.
\end{eqnarray*}

{\it  Case 3:  $0<\theta\leq\dfrac{1}{L-1}$.}  In this case,
 $f^\prime (x)\geq 0$ for $0\leq x\leq L-1$. Hence
\begin{equation}\label{5.2.10}
ie^{-\theta i}\leq \int_i^{i+1} xe^{-\theta x}dx,\; i=1,\dots,L-2.
\end{equation}
Therefore,
\begin{eqnarray*}
S_1&=&\sum_{i=1}^{L-2}ie^{-\theta i}+(L-1)e^{-(L-1)\theta}
\leq (L-1)e^{-(L-1)\theta}+\sum_{i=1}^{L-2}\int_i^{i+1} xe^{-\theta x}dx\\
&=&(L-1)e^{-(L-1)\theta}+\int_{1}^{L-1} xe^{-\theta x}dx
\end{eqnarray*}
Combining the three cases discussed above gives
\begin{equation}\label{5.2.12}
S_1\leq M_{\theta,L} +\int_{1}^{L-1} xe^{-\theta
x}dx=M_{\theta,L}+\dfrac{1}{\theta}\left(1+\dfrac{1}{\theta}\right)e^{-\theta}
-\dfrac{1}{\theta}\left(L-1+\dfrac{1}{\theta}\right)e^{-\theta(L-1)},
\end{equation}
where $M_{\theta,L}$ is defined in \dref{5.2.11}.

An analogous but simpler  analysis of $S_2,\; S_3$, and $S_4$ gives
\begin{eqnarray}\label{5.2.13}
\nonumber
&&S_2\leq 1+\dfrac{1}{\theta}-\dfrac{1}{\theta}e^{-\theta(L-1)},\\
&&S_3\leq\dfrac{1}{\theta}e^{-\theta
(N-2)}\left[\left(L+1-\dfrac{1}{\theta}\right)e^{\theta L}
+\dfrac{1}{\theta}-1\right]\triangleq M_{\theta,N},\\
\nonumber &&S_4\leq \dfrac{L}{\theta}\left[e^{-\theta
(L-1)}-e^{-\theta (N-L-2)}\right].
\end{eqnarray}

We next show when $N>9$,
\begin{equation}\label{5.3.14}
M_{\theta,N}<\dfrac{1}{\theta^2}e^{-\theta(L-1)}+\dfrac{L}{\theta}e^{-\theta(N-L-2)}.
\end{equation}
Since $L=N/3$ or $L=\lfloor N/3\rfloor+1$ if $N/3$ is not an
integer, it follows that $L\geq 4$ when $N>9$, hence
\begin{equation}\label{5.3.15}
\begin{array}{ll}
M_{\theta,N}-\dfrac{1}{\theta^2}e^{-\theta(L-1)}-\dfrac{L}{\theta}e^{-\theta(N-L-2)}
&=\disp
\dfrac{1}{\theta^2}e^{-\theta(N-2)}\left[(\theta-1)\left(e^{\theta
L}-1\right)-e^{\theta (N-L-1)}\right]\crr &\leq \disp
\dfrac{1}{\theta^2}e^{-\theta(N-2)}\left[(\theta-1)\left(e^{\theta
L}-1\right)-e^{\theta (2L-3)}\right]\crr &<\disp
\dfrac{1}{\theta^2}e^{-\theta(N-2)}\left(e^{\theta
L}-1\right)\left[\theta-e^{(L-3)\theta}\right]\crr &< 0.
\end{array}
\end{equation}
As a consequence,
\begin{eqnarray*}
S_1+S_2+S_3+S_4&\leq&M_{\theta,L}+\dfrac{1}{\theta}\left(1+\dfrac{1}{\theta}\right)e^{-\theta}
-\dfrac{1}{\theta}\left(L-1+\dfrac{1}{\theta}\right)e^{-\theta(L-1)}+1+\dfrac{1}{\theta}\\
&&-\dfrac{1}{\theta}e^{-\theta(L-1)}
+M_{\theta,N}+\dfrac{L}{\theta}\left[e^{-\theta (L-1)}-e^{-\theta
(N-L-2)}\right]\\
&=&M_{\theta,L}
+\dfrac{1}{\theta}\left(1+\dfrac{1}{\theta}\right)e^{-\theta}-\dfrac{1}{\theta^2}e^{-\theta
(L-1)}+1+\dfrac{1}{\theta}+M_{\theta,N}-\dfrac{L}{\theta}e^{-\theta
(N-L-2)}\\
&<&M_{\theta,L}+\left(1+\dfrac{1}{\theta}\right)^2.
\end{eqnarray*}
Hence
\begin{eqnarray*}
\|Y_0-X_0\|_F&<&\left(\sqrt{2}+\dfrac{1}{4M\pi^2\alpha_0
T_1}\right)M_0e^{-\alpha_0
M^2\pi^2 T_1}\sqrt{S_1+S_2+S_3+S_4}\\
&<&\left(\sqrt{2}+\dfrac{1}{4M\pi^2\alpha_0
T_1}\right)M_0e^{-\alpha_0 M^2\pi^2
T_1}\sqrt{M_{\theta,L}+\left(1+\dfrac{1}{\theta}\right)^2}.
\end{eqnarray*}
By almost the same analysis to $\|Y_1-X_1\|_F$, we  can achieve the
estimation \dref{5.1.6}. The details are  omitted. This completes
the proof of the theorem.
\end{proof}

The next lemmas show the effect of perturbations in a matrix to its
generalized inverse or eigenvalues.
\begin{lemma}\label{Lem5.1}(\cite{Stewart})
For any two matrices $A$ and $B$ with $B=A+E$, if
$\rank(A)=\rank(B)$, then
\begin{equation}\label{5.1.11}
\left\|B^\dag-A^\dag\right\|_2\leq \dfrac{1+\sqrt{5}}{2}
\left\|A^\dag\right\|_2\cdot\left\|B^\dag\right\|_2\cdot
\left\|E\right\|_2,
\end{equation}
where $\|\cdot\|_2$ denotes  the matrix spectral norm (matrix
2-norm).
\end{lemma}

\begin{lemma}\label{Lem5.2}(\cite{Wedin})
If
\begin{equation}\label{5.1.12}
\rank(A+E)=\rank(A) \;\; and \;\;
\left\|E\right\|_2<\dfrac{1}{\left\|A^\dag\right\|_2},
\end{equation}
then
\begin{equation}\label{5.1.13}
\left\|(A+E)^\dag\right\|_2 \leq
\dfrac{\left\|A^\dag\right\|_2}{1-\left\|A^\dag\right\|_2\cdot\|E\|_2}.
\end{equation}
\end{lemma}

\begin{lemma}\label{Lem5.3}(\cite{Bauer})
If $A$ is diagonalizable, i.e.,
$$A=X\Lambda X^{-1},\; where \; \Lambda=\diag (\lambda_1,\dots,\lambda_n),
$$
then for any $\widetilde{\lambda}\in \lambda(\widetilde{A})$, there
exists a $\lambda\in \lambda(A)$ such that
\begin{equation}\label{5.1.14}
|\widetilde{\lambda}-\lambda|\leq
\kappa(X)\cdot\left\|\widetilde{A}-A\right\|_2,
\end{equation}
where $\lambda(A)$ is the set of the eigenvalues of  $A$ and
$\kappa(X)$ is the (spectral) condition number of  $X$,   defined as
$$\kappa(X)=\|X\|_2\cdot\left\|X^{-1}\right\|_2.$$
\end{lemma}

Suppose the singular values of $Y_0$ are $\sigma(Y_0)=\{\sigma_i\}$
and $Y_{0,M}$ is the rank-$M$ truncated approximation of $Y_0$
defined by
\begin{equation}\label{5.1.16}
Y_{0,M}=U_{0,M}AV_{0,M}^\top,
\end{equation}
where $U_{0,M},\; A$, and $V_{0,M}$ are defined in \dref{4.2.10.2}.
Now we are in a position to give an error analysis for the infinite
spectral estimation problem \dref{5.1.1} using the matrix pencil
method.
\begin{theorem}\label{Th5.2.2}
Let $\sigma_1\geq\sigma_2\geq\dots\geq\sigma_M$ be the first $M$
singular values of the matrix $Y_0$, and assume that $Y_{0,M}^\dag
Y_1$ is diagonalizable, i.e. $Y_{0,M}^\dag
Y_1=X_M\widetilde{\Lambda}_M X_M^{-1},$ where
$$\widetilde{\Lambda}_M=\diag
(\widetilde{z}_1,\dots,\widetilde{z}_M,0,\dots,0),\;
\widetilde{z}_1\geq \widetilde{z}_2\geq \dots\geq \widetilde{z}_M.
$$ The nonzero eigenvalues of the matrix $X_0^\dag X_1$ are supposed
to be
$$\Lambda_M=\{z_1,z_2,\dots,z_M\},\; z_1\geq z_2\geq \dots\geq z_M.$$
Let $\theta$ and $M_{\theta,L}$ be defined as in \dref{5.2.1},
\dref{5.2.11}, respectively, and let
\begin{equation}\label{5.1.17}
\rho=\left(\left\|Y_{0,M}-Y_0\right\|_2+\left(\sqrt{2}+\dfrac{1}{4M\pi^2\alpha_0
T_1}\right)M_0e^{-\alpha_0 M^2\pi^2
T_1}\sqrt{M_{\theta,L}+\left(1+\dfrac{1}{\theta}\right)^2}\right)\bigg/\sigma_M.
\end{equation}
If $\rho<1$, then
\begin{equation}\label{5.1.19}
\begin{array}{ll}
|\widetilde{z}_n-z_n|<&\disp\dfrac{\kappa (X_M)}{\sigma_M\cdot
(1-\rho)}\cdot \Bigg[\dfrac{1+\sqrt{5}}{2}\rho\|Y_1\|_2
+\disp\left(\sqrt{2}+\dfrac{1}{4M\pi^2\alpha_0
T_1}\right)M_0e^{-\alpha_0 M^2\pi^2 T_1} \crr & \times
\sqrt{M_{\theta,L+1}+\dfrac{1}{\theta}
\left(1+\dfrac{1}{\theta}\right)e^{-\theta}}\Bigg].
\end{array}
\end{equation}
In particular, if $\theta>\dfrac{1}{L-1}$, then
\begin{equation}\label{5.1.10}
|\widetilde{z}_n-z_n|<\dfrac{\kappa (X_M)\cdot \rho}{\sigma_M\cdot
(1-\rho)}\cdot \left[\dfrac{1+\sqrt{5}}{2}\|Y_1\|_2+\sigma_M\right].
\end{equation}
\end{theorem}
\begin{proof} We need to estimate the matrix norm $\left\|Y_{0,M}^\dag
Y_1-X_0^\dag X_1\right\|_2$ first, which is done as follows:
\begin{eqnarray*}
\left\|Y_{0,M}^\dag Y_1-X_0^\dag X_1\right\|_2&=&\left\|Y_{0,M}^\dag Y_1-X_0^\dag Y_1+X_0^\dag Y_1-X_0^\dag X_1\right\|_2\\
&\leq&\left\|Y_{0,M}^\dag Y_1-X_0^\dag Y_1\right\|_2+\left\|X_0^\dag Y_1-X_0^\dag X_1\right\|_2\\
&\leq&\left\|Y_{0,M}^\dag-X_0^\dag\right\|_2\cdot\|Y_1\|_2+\left\|X_0^\dag\right\|_2\cdot\|Y_1-X_1\|_2.
\end{eqnarray*}
Since  $Y_{0,M}$ is the rank-$M$ truncated matrix of $Y_0$,
$\rank(X_0)=M=\rank(Y_{0,M})$. An application of Lemma \ref{Lem5.1}
yields
\begin{equation}\label{5.1.21}
\left\|Y_{0,M}^\dag-X_0^\dag\right\|_2\leq\dfrac{1+\sqrt{5}}{2}
\left\|Y_{0,M}^\dag\right\|_2\cdot\left\|X_0^\dag\right\|_2\cdot
\left\|Y_{0,M}-X_0\right\|_2.
\end{equation}
So
\begin{equation}\label{5.1.22}
\left\|Y_{0,M}^\dag Y_1-X_0^\dag X_1\right\|_2\leq
\dfrac{1+\sqrt{5}}{2}\left\|Y_{0,M}^\dag\right\|_2\cdot\left\|X_0^\dag\right\|_2\cdot
\|Y_{0,M}-X_0||_2\cdot \|Y_1\|_2+\left\|X_0^\dag\right\|_2\cdot
\|Y_1-X_1\|_2.
\end{equation}
Since  $\left\|Y_{0,M}^\dag\right\|_2=\dfrac{1}{\sigma_M}$,
$$
\begin{array}{ll}
\disp  \left\|Y_{0,M}-X_0\right\|_2\cdot
\left\|Y_{0,M}^\dag\right\|_2&\leq \disp
\left(\left\|Y_{0,M}-Y_0\right\|_2+\left\|Y_0-X_0\right\|_2\right)\cdot
\left\|Y_{0,M}^\dag\right\|_2\crr &\leq \disp
\left(\left\|Y_{0,M}-Y_0\right\|_2+\left\|Y_0-X_0\right\|_F\right)\cdot
\left\|Y_{0,M}^\dag\right\|_2\crr  &<\rho,
\end{array}
$$
where the last inequality is based on the estimation of
$\|Y_0-X_0\|_F$ in Theorem \ref{Th5.1.1}. Since $\rho<1$, it follows
from Lemma \ref{Lem5.2} that
\begin{equation}\label{5.1.23}
\left\|X_0^\dag\right\|_2\leq
\dfrac{\left\|Y_{0,M}^\dag\right\|_2}{1-\left\|Y_{0,M}^\dag\right\|_2\cdot
\left\|Y_{0,M}-X_0\right\|_2}<\dfrac{1}{(1-\rho)\sigma_M}.
\end{equation}
As a result,
\begin{eqnarray*}
\left\|Y_{0,M}^\dag Y_1-X_0^\dag X_1\right\|_2&\leq&
\left\|X_0^\dag\right\|_2\cdot\left[\dfrac{1+\sqrt{5}}{2}\left\|Y_{0,M}^\dag\right\|_2\cdot
\|Y_{0,M}-X_0||_2\cdot \|Y_1\|_2+\|Y_1-X_1\|_2\right]\\
&<&\dfrac{1}{(1-\rho)\cdot\sigma_M}\left[\dfrac{1+\sqrt{5}}{2}\rho\|Y_1\|_2+\|Y_1-X_1||_F\right].
\end{eqnarray*}
By Lemma \ref{Lem5.3}, we have
\begin{eqnarray*}
|\widetilde{z}_n-z_n|&\leq&\kappa (X_M)\cdot\left\|Y_{0,M}^\dag
Y_1-X_0^\dag X_1\right\|_2 <\dfrac{\kappa
(X_M)}{(1-\rho)\cdot\sigma_M}\left[\dfrac{1+\sqrt{5}}{2}\rho\|Y_1\|_2+\|Y_1-X_1||_F\right]\\
&<&\dfrac{\kappa (X_M)}{\sigma_M\cdot (1-\rho)}\cdot
\Bigg[\dfrac{1+\sqrt{5}}{2}\rho\|Y_1\|_2
+\left(\sqrt{2}+\dfrac{1}{4M\pi^2\alpha_0 T_1}\right)M_0e^{-\alpha_0
M^2\pi^2 T_1}\crr &&\times\sqrt{M_{\theta,L+1}
+\dfrac{1}{\theta}\left(1+\dfrac{1}{\theta}\right)e^{-\theta}}\Bigg].
\end{eqnarray*}
In particular, if $\theta>\dfrac{1}{L-1}$, it follows from the
definition of $M_{\theta,L}$ in \dref{5.2.11} that
\begin{equation}\label{5.1.25}
M_{\theta,L+1}=M_{\theta,L},\;\;\theta>\dfrac{1}{L-1}.
\end{equation}
On the other hand, since
$$\dfrac{1}{\theta}\left(1+\dfrac{1}{\theta}\right)e^{-\theta}<\left(1+\dfrac{1}{\theta}\right)^2,\;\;\theta>0,
$$
it follows from  \dref{5.1.25} that
$$\sqrt{M_{\theta,L+1}
+\dfrac{1}{\theta}\left(1+\dfrac{1}{\theta}\right)e^{-\theta}}<\sqrt{M_{\theta,L}
+\left(1+\dfrac{1}{\theta}\right)^2},$$ and then
\begin{eqnarray*}
|\widetilde{z}_n-z_n|&<&\dfrac{\kappa (X_M)}{\sigma_M\cdot
(1-\rho)}\cdot
\Bigg[\dfrac{1+\sqrt{5}}{2}\rho\|Y_1\|_2+\left(\sqrt{2}+\dfrac{1}{4M\pi^2\alpha_0
T_1}\right)M_0e^{-\alpha_0 M^2\pi^2 T_1}\\
&&\times\sqrt{M_{\theta,L}
+\left(1+\dfrac{1}{\theta}\right)^2}\Bigg]\\
&=&\dfrac{\kappa (X_M)}{\sigma_M\cdot (1-\rho)}\cdot
\bigg[\dfrac{1+\sqrt{5}}{2}\rho\|Y_1\|_2+\rho\cdot\sigma_M-\|Y_{0,M}-Y_0\|_2\bigg]\\
&<&\dfrac{\kappa (X_M)\cdot \rho}{\sigma_M\cdot (1-\rho)}\cdot
\left[\dfrac{1+\sqrt{5}}{2}\|Y_1\|_2+\sigma_M\right].
\end{eqnarray*}
This ends the proof of the theorem.
\end{proof}
\begin{remark}\label{Re5.1}
By $z_n=e^{-\lambda_n T_s}$, we can also obtain an error estimation
$|\widetilde{\lambda}_n-\lambda_n|$, between the estimated
eigenvalues and the exact eigenvalues, that is (for
$\theta>\frac{1}{L-1}$),
\begin{equation}\label{5.1.26}
|\widetilde{\lambda}_n-\lambda_n|=\dfrac{|\ln\widetilde{z}_n-\ln
z_n|}{T_s}=\dfrac{|\widetilde{z}_n-z_n|}{T_s\cdot\bar{z}_n}
<\dfrac{\kappa (X_M)\cdot \rho}{\sigma_M
(1-\rho)T_s\cdot\bar{z}_n}\cdot
\left[\dfrac{1+\sqrt{5}}{2}\|Y_1\|_2+\sigma_M\right],
\end{equation}
where the mean value theorem has been applied in the second equality
and $\bar{z}_n$ is between $\widetilde{z}_n$ and $z_n$. In addition,
we can choose $\bar{z}=\widetilde{z}_n$ in  case of
$|\widetilde{z}_n-z_n|\ll|\widetilde{z}_n|$.
\end{remark}
\begin{remark}\label{Re5.2} We point out that the estimation seems
hard to improve further.  It can be seen that  the estimation of
$\left\|Y_{0,M}^\dag Y_1-X_0^\dag X_1\right\|_2$ plays a key role in
the proof of Theorem \ref{Th5.2.2}. The condition $\rho<1$ is mainly
for the estimation of $\left\|X_0^\dag\right\|$, which becomes
extremely complicated for  $\rho\geq1$ due to  the unknown nature of
$X_0$. However, by  \dref{5.1.17}, since the value of $\rho$ is
determined by $M$ (determined by $\varepsilon$ in \dref{4.2.8}) and
$T_1$, the parameters $\varepsilon$ and $T_1$ can be chosen
appropriately in applications to make $\rho$ relatively small, and
from \dref{5.1.10}, the error bound becomes smaller as
$\rho/\sigma_M$ becomes smaller.
\end{remark}

\section{Numerical simulation}

In this section, we present some numerical examples to illustrate
the performance of the algorithm developed in section 3. It might be
worth noting that all the calculated numbers in this section are
rounded to four digits after the decimal point.

First, to generate data for the inverse  process, we take a real
diffusivity $\alpha^\ast$ and an initial value $u_0^\ast(x)$ to
solve the direct problem to obtain the values of observation data
$y(t)=u(0,t;f,u_0^\ast)$ over an interval $(0,T_3]$. In this
experiment, we take $\alpha=\alpha^\ast=4$ and
$$u_0(x)=u_0^\ast(x)=x-9\cos \pi x+5\cos 3\pi x,$$ in
system \dref{1.1}. Since
$$
\langle u_0^\ast,\phi_{2n}\rangle=0,\;n=1,2,\ldots,
$$
this initial value is not generic (\cite{Suzuki1}). The time
interval is chosen to be $[T_1,T_2,T_3]=[0.3,0.8,1.3]$, and the
control function $f(t)$ is chosen to be that defined in
\dref{3.14.3}. Then the observation data can be obtained from
\dref{2.2}-\dref{2.11}.

Now we assume that both the real value of the diffusion coefficient
$\alpha^\ast$ and initial value $u_0^\ast(x)$ of system \dref{1.1}
are unknown, and the only known information for $\alpha^\ast$ and
$u_0^\ast(x)$ is that
\begin{equation}\label{6.1.9}
\alpha^\ast\geq\alpha_0=3,\; \|u_0^\ast\|_{L^2(0,1)}\leq M_0=15.
\end{equation}
We will treat the measured value $y(t)$ as the inverse dynamical
data, and try to reconstruct the unknown $\alpha^\ast$ and
$u_0^\ast(x)$ by the proposed  algorithm.

{\it Step 1: Estimate $\{\widetilde{\lambda}_{n_k}\}_{k=0}^{M-1}$
from the measured value at every sampling time by the matrix pencil
method.}

Let $N_1=50$ and $0.3=t_0<t_1<\cdots <t_{50}=0.8$ be the equidistant
sample points with sampling period $T_s=0.01$. The pencil parameter
$L=17,$ and the number of exponential components $M=2$ which is
obtained from \dref{4.2.8}, where the threshold
$\varepsilon=10^{-10}$. The estimated
$\left\{\widetilde{z}_{n_k},\widetilde{\lambda}_{n_k}\right\}_{k=0}^1$
by virtue of the matrix pencil method are shown in Table
\ref{Tab1.1} and \ref{Tab1.2}, where
$\widetilde{z}_{n_k}=e^{-\widetilde{\lambda}_{n_k}T_s}.$

\begin{table}[H]
\caption{The estimated
$\left\{\widetilde{z}_{n_k},\;\widetilde{\lambda}_{n_k},\;\widetilde{C}_{n_k}\right\}_{k=0}^1$}
\centering \subtable[$\left\{\widetilde{z}_{n_k}\right\}_{k=0}^1$]{
\begin{tabular}{c c c}
\toprule
$k$ & 0 & 1 \\
\midrule
$\widetilde{z}_{n_k}$&1.0000&0.6738\\
\bottomrule
\end{tabular}
       \label{Tab1.1}
} \qquad
\subtable[$\left\{\widetilde{\lambda}_{n_k}\right\}_{k=0}^1$]{
\begin{tabular}{c c c}
\toprule
$k$ & 0 & 1 \\
\midrule
$\widetilde{\lambda}_{n_k}$&0.0000&39.4784\\
\bottomrule
\end{tabular}
\label{Tab1.2}} \qquad
\subtable[$\left\{\widetilde{C}_{n_k}\right\}_{k=0}^1$]{
\begin{tabular}{c c c}
\toprule
$k$ & 0 & 1 \\
\midrule
$\widetilde{C}_{n_k}$&0.5000&-9.4077\\
\bottomrule
\end{tabular}
\label{Tab1.3}}
\end{table}

{\it Step 2: Estimate $\left\{\widetilde{C}_{n_k}\right\}_{k=0}^1$
by solving the following linear least square problem:}
\begin{equation}\label{5.9}
\left\{\widetilde{C}_{n_k}\right\}_{k=0}^1=\argmin
\sum_{i=0}^{49}\left[y_i-\sum_{k=0}^1\widetilde{C}_{n_k}
e^{-\widetilde{\lambda}_{n_k}t_i}\right]^2.
\end{equation}
The estimated $\left\{\widetilde{C}_{n_k}\right\}_{k=0}^1$ are shown
in Table \ref{Tab1.3}.

It has been stated in Remark \ref{Re4.3.2} that
$$u(0,t;0,u_0^\ast)\approx \widetilde{y}(t)=0.5000-9.4077e^{-39.4784t},\; t>0.$$

{\it Step 3: Estimate the approximation of  $\alpha$.}

Similar to {\it Step 1}, let $N_2=50$ and let $0.8=t_0<t_1<\cdots
<t_{50}=1.3$ be the equidistant sample points with sampling period
$T_s^\prime=0.01$. Then  the pencil parameter $L^\prime=17,$ and the
number of exponential components $M^\prime=5$, where the threshold
$\varepsilon=10^{-10}$. The estimated
$\left\{C_n^\prime,\lambda_n^\prime\right\}_{n=0}^4$ are listed in
Table \ref{Tab2}.

It is shown in Remark \ref{Re4.3.3} that the pairs
$(C_n^\prime,\lambda_n^\prime)$ that satisfy \dref{4.30} are more
credible to estimate $\alpha$. It is obvious from Table \ref{Tab2}
that $\lambda_1^\prime$ and $\lambda_2^\prime$ are more suitable to
estimate $\alpha$, which can be recovered from \dref{4.28} that
$\alpha\approx\widetilde{\alpha}=4.0000$. In fact, $\alpha$ can also
be estimated by the $\alpha_0$ in Table \ref{Tab2} which is obtained
from $C_0^\prime$ in \dref{4.26} by
\begin{equation}\label{5.11}
\alpha\approx\alpha_0=-\dfrac{1}{3C_0^\prime}=4.0000.
\end{equation}

\begin{table}[H]
\caption{The estimated
$\left\{C_n^\prime,\lambda_n^\prime\right\}_{n=0}^4$ and the
estimated $\widetilde{\alpha}$}\label{Tab2}
\begin{center}
\begin{tabular}{c c c c c c}
\toprule
$n$&0&1&2&3&4 \\
\midrule
$100*C_n^\prime$&-8.3333&5.0661&1.2665&0.5664&1.4090\\
\midrule
$100*\lambda_n^\prime$&0.0000&39.4784&157.9137&355.5370&790.8813\\
\midrule
$100*C_n^\prime*\lambda_n^\prime$&/&2.0000&2.0000&2.0139&11.1438\\
\midrule
$\alpha\approx\dfrac{\lambda_n^\prime}{T_s^\prime n^2\pi^2}$&$\alpha_0$&4.0000&4.0000&4.0026&5.0083\\
\bottomrule
\end{tabular}
\end{center}
\end{table}

{\it Step 4: Estimate $\alpha$ from $\widetilde{\lambda}_{n_1}$ and
reconstruct $u_0(x)$.}

After obtaining the estimations $\alpha\approx 4.0000$ and
$\left\{\widetilde{\lambda}_{n_k}\right\}$ in Table \ref{Tab1.2}, we
can determine the series
$\K_M=\left\{n_k\right\}_{k=0}^{M-1}=\{0,1\}$ by \dref{4.31}.
Actually, the coefficient estimation  through \dref{4.1.37} is also
equal to $\widetilde{\alpha}=4.0000$.

Next we can estimate $u_0(x)$ by solving the matrix equation
\dref{4.35.2} with TSVD, where $T_0=0.01$ and $\widetilde{M}=20$.
The corresponding GCV analysis is shown in Figure \ref{Fig2.1}, from
which the regularization parameter is found to be $k=6$. Then the
solution of \dref{4.35.2} is given by \dref{4.35.6} and $u_0(x)$ can
be estimated by the Fourier series expansion:
\begin{equation}\label{5.11}
u_0(x)\approx\sum_{n=0}^{19} A_n(0) \cos n\pi x.
\end{equation}
The results are given in Figure \ref{Fig2.2}, from which we can see
that the estimated initial value is in agreement with the real one.

\begin{figure}[H]
 \centering
\subfigure[The GCV function]{
 \label{Fig2.1}
 \includegraphics[scale=0.35]{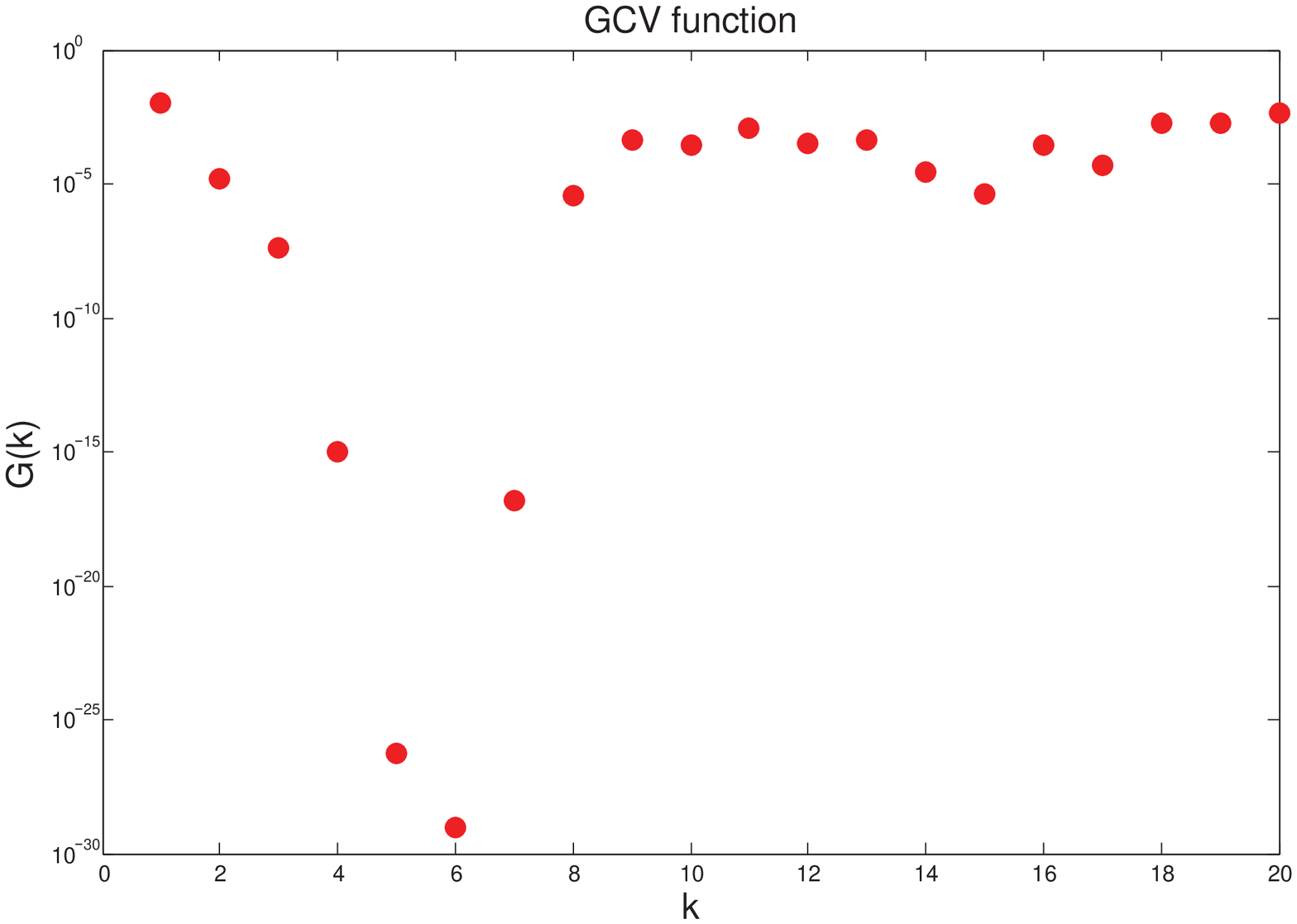}}\;\;
 \subfigure[The initial value]{
  \label{Fig2.2}
 \includegraphics[scale=0.35]{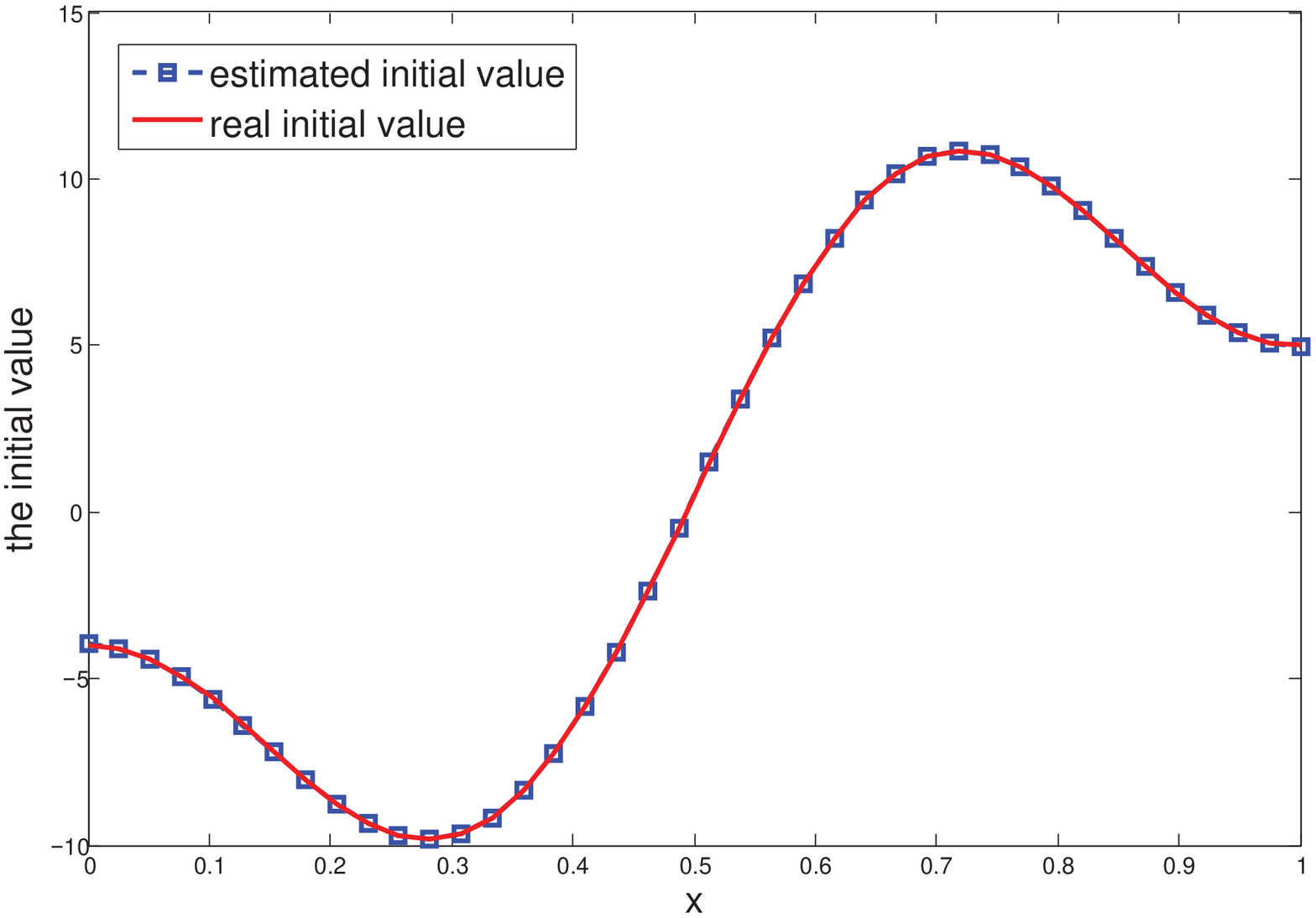}}
\caption{The GCV function and initial value} \label{Fig-3}
 \end{figure}

Finally, based on the error analysis in the previous section, we can
give the bound of the error between the real value $\alpha$ and the
reconstructed one $\widetilde{\alpha}=4.0000$. The parameters that
are relevant to the error analysis are given in Table \ref{Tab6}.

\begin{table}[H] \caption{The parameters for error analysis}\label{Tab6} \centering
\begin{tabular}{c c c c c c c c c c}\\
\toprule $M_0$ & $\alpha_0$ & $M$ & $N$ & $L$ & $T_1$ & $T_s$ &
$\theta$ & $M_{\theta,L}$ & $\|Y_1\|_2$\\
\midrule
15 & 3 & 2 & 50 & 17 & 0.3 & 0.01 & 2.3687 & 0.0936 & 11.8427\\
\bottomrule
\end{tabular}
\begin{tabular}{c c c c c}\\
\toprule $\sigma_M$ & $\|Y_{0,M}-Y_0\|_2$ & $\kappa(X_M)$ & $\rho$ \\
\midrule $9.5089\times 10^{-5}$ & $2.2494\times 10^{-15}$ & 17.9467 & $1.4522\times 10^{-10}$\\
\bottomrule
\end{tabular}
\end{table}

By Theorem \ref{Th5.2.2},
\begin{equation}\label{5.12}
|\widetilde{z}_n-z_n|<\dfrac{\kappa (X_M)\cdot \rho}{\sigma_M\cdot
(1-\rho)}\cdot
\left[\dfrac{1+\sqrt{5}}{2}\|Y_1\|_2+\sigma_M\right]\approx
5.2521\times 10^{-4}\triangleq e,\; n\in \K_M=\{0,1\}.
\end{equation}
It is clear from Table \ref{Tab1.1} that
$\widetilde{z}_1=\widetilde{z}_{n_1}=0.6738\gg e,$ and it follows
from \dref{5.1.26} in Remark \ref{Re5.1} that
\begin{equation}\label{5.13}
|\alpha-\widetilde{\alpha}|=\dfrac{|\lambda_1-\widetilde{\lambda}_1|}{\pi^2}<\dfrac{e}{\pi^2\cdot
T_s\cdot\widetilde{z}_1}\approx 7.8974\times 10^{-3}.
\end{equation}
We thus know that the real diffusion coefficient $\alpha^\ast$ is
between 3.9921 and 4.0079.

\section{Concluding remarks}

In this paper, we represent the boundary observation with boundary
Neumann control for a one-dimensional heat equation into a Dirichlet
series in terms of spectrum determined by the diffusivity and
coefficients determined by the initial value. The identification of
diffusion coefficient and initial value is therefore transformed
into an inverse problem of reconstruction of spectrum-coefficient
data from the observation. Taking the first finite terms of the
series, the problem happens to be an inverse problem of finite
exponential sequence with deterministic small perturbation. We are
thus able to develop an algorithm to reconstruct simultaneously the
diffusion coefficient and initial value by the matrix pencil method
which is used in signal processing. An error analysis is presented
and a numerical experiment is  carried out to validate the
efficiency and accuracy of the proposed algorithm. The method
developed is promising and can be applied in identification of
variable coefficients and other PDEs.

\section*{Acknowledgements}

This work was supported by the National Natural Science Foundation
of China and the National Research Foundation of South Africa.


\begin{thebibliography}{20}

\bibitem{Abramowitz}
M. Abramowitz and I.A. Stegun, {\it Handbook of Mathematical
Functions with Formulas, Graphs, and Mathematical Tables}, Dover
Publications,  New York, 1964.

\bibitem{Bauer}
F.L. Bauer and C.T. Fike, Norms and exclusion theorems, {\it  Numer.
Math.},  2(1960), 137--141.

\bibitem{Benabdallah1}
A. Benabdallah, P. Gaitan,  and J.L. Rousseau, Stability of
discontinuous diffusion coefficients and initial conditions in an
inverse problem for the heat equation, {\it SIAM J. Control Optim.},
46(2007), 1849-1881.

\bibitem{Chang}
J.D. Chang and B.Z. Guo, Identification of variable spacial
coefficients for a beam equation from boundary measurements, {\it
Automatica,} 43(2007), 732--737.

\bibitem{Choulli1}
M. Choulli and M. Yamamoto, Uniqueness and stability in determining
the heat radiative coefficient, the initial temperature and a
boundary coefficient in a parabolic equation, {\it Nonlinear Anal.},
69(2008), 3983-3998.

\bibitem{Deng1}
Z.C. Deng, L. Yang and J.N. Yu, Identifying the radiative
coefficient of heat conduction equations from discrete measurement
data, {\it Appl. Math. Lett.}, 22(2009), 495-500.

\bibitem{Golub}
G.H. Golub, M. Heath and G. Wahba, Generalized cross-validation as a
method for choosing a good ridge parameter, {\it Technometrics},
21(1979), 215-223.

\bibitem{Guo1}
B.Z. Guo and  J.D. Chang, Simultaneous identifiability of
coefficients, initial state and source for string and beam equations
via boundary control and observation, {\it Proc. 8th Asian Control
Conference,} Kaohsiung, 2011, 365--370.

\bibitem{Gutman}
S. Gutman and  J.H. Ha, Identifiability of piecewise constant
conductivity in a heat conduction process, {\it SIAM J. Control
Optim.}, 46(2007), 694--713.

\bibitem{Hansen}
P.C. Hansen, {\it Discrete Inverse Problems: Insight and
Algorithms}, SIAM, Philadelphia, 2010.

\bibitem{Hua1}
Y.B. Hua and  T.K. Sarkar, Further analysis of three modern
techniques for pole retrieval from data sequence, {\it Proc. 30th
Midwest Symp. Circuits Syst.,} Syracuse, NY, Aug. 1987, 793-797.

\bibitem{Hua2}
Y.B. Hua and T.K. Sarkar, Matrix pencil method and its performance,
{\it Proc. IEEE Int. Conf. Acoust., Speech, Signal Processing,} NY,
Apr. 1988, 2476-2479.

\bibitem{Hua3}
Y.B. Hua and T.K. Sarkar, Matrix pencil method for estimating
parameters of exponentially damped/undamped sinusoids in noise, {\it
IEEE Trans.  Acoust.  Speech Signal Process.,} 38(1990), 814--824.

\bibitem{Isakov1}
V. Isakov, {\it Inverse Problems for Partial Differential
Equations}, Springer, New York, 1998.

\bibitem{Kirsch1}
A. Kirsch, {\it An Introduction to the Mathematical Theory of
Inverse Problems}, Springer, New York, 1999.

\bibitem{Kitamura}
S. Kitamura and S. Nakagiri, Identifiability of spatially-varying
and constant parameters in distributed systems of parabolic type,
{\it SIAM J. Contorl Optim.}, 15(1977), 785--802.

\bibitem{Levitan}
B.M. Levitan, {\it Inverse Sturm-Liouville Problems}, VNU Science
Press, Utrecht, 1987.

\bibitem{Lorenzi1}
A. Lorenzi, Identification of the thermal conductivity in the
nonlinear heat equation, {\it Inverse Problems}, 3(1987), 437-451.

\bibitem{Ma}
Y.J. Ma, C.L. Fu,  and Y.X. Zhang, Identification of an unknown
source depending on both time and space variables by a variational
method, {\it Appl. Math. Model.}, 36(2012), 5080--5090.

\bibitem{Mirsky}
L. Mirsky, Symmetric gauge functions and unitarily invariant norms,
{\it  Quart. J. Math. Oxford Ser. (2)},  11(1960), 50--59.

\bibitem{Murayama}
R. Murayama, The Gel'fand-Levitan theory and certain inverse
problmes for the parabolic equation, {\it J. Fac. Sci. Univ. Tokyo
Sect. IA Math.}, 28(1981), 317--330.

\bibitem{Nakagiri}
S. Nakagiri, Identifiability of linear systems in Hilbert spaces,
{\it SIAM J. Contorl Optim.}, 21(1983), 501--530.

\bibitem{Orlov}
Y. Orlov and J. Bentsman, Adaptive distributed parameter systems
identification with enforceable identifiability  conditions and
reduced-order spatial differentiation, {\it IEEE Trans. Automat.
Control}, 45(2000), 203--216.

\bibitem{Pierce}
A. Pierce, Unique identification of eigenvalues and coefficients in
a parabolic equation, {\it SIAM J. Contorl Optim.}, 17(1979),
494--499.

\bibitem{Poschel}
J. P\"oschel and E. Trubowitz, {\it Inverse Spectral Theory},
Academic Press, Orlando, 1987.

\bibitem{Ramdani}
K. Ramdani, M. Tucsnak,  and G. Weiss, Recovering the initial state
of an infinite-dimensional system using observers, {\it Automatica,}
46(2010), 1616--1625.

\bibitem{Smyshlyaev}
A. Smyshlyaev, Y. Orlov and M. Krstic, Adaptive identification of
two unstable PDEs with boundary sensing and actuation, {\it Int. J.
Adapt. Control Signal Process.}, 23(2009), 131--149.

\bibitem{Stewart}
G.W. Stewart, On the perturbation of pseudo-inverses, projections
and linear least squares problems, {\it SIAM Rev.},  19(1977),
634--662.

\bibitem{Suzuki1}
T. Suzuki and R. Murayama, A uniqueness theorem in an identification
problem for coefficients of parabolic equations, {\it Proc. Japan
Acad. Ser. A Math. Sci.}, 56(1980), 259--363.

\bibitem{Suzuki2}
T. Suzuki, Uniqueness and nonuniqueness in an inverse problem for
the parabolic equation, {\it J. Differential Equations}, 47(1983),
296--316.

\bibitem{Titchmarsh}
E.C. Titchmarsh, {\it Introduction to the Theory of Fourier
Integrals}, 2nd Edtion, Clarendon Press, Oxford, 1948.


\bibitem{Wang1}
Y.B. Wang, J. Cheng, J. Nakagawa, and M. Yamamoto, A numerical
method for solving the inverse heat conduction problem without
initial value, {\it Inverse Probl. Sci. Eng.},  18(2010), 655--671.


\bibitem{Wedin}
P.A. Wedin, Perturbation theory for pseudo-inverses, {\it BIT,}
13(1973), 217--232.


\bibitem{Xu}
G.Q. Xu, State reconstruction of a distributed parameter system with
exact observability, {\it  J. Math. Anal. Appl.},  409(2014),
168--179.


\bibitem{Yamamoto1}
M. Yamamoto and J. Zou, Simultaneous reconstruction of the initial
temperature and heat radiative coefficient, {\it Inverse Problems,}
17(2001), 1181--1202.

\bibitem{Yamamoto2}
M. Yamamoto, Carleman estimates for parabolic equations and
applications, {\it Inverse Problems,} 25(2009), 123013 (75pp).

\bibitem{Zheng1}
G.H. Zheng and T. Wei, Recovering the source and initial value
simultaneously in a parabolic equation, {\it Inverse Problems,}
30(2014),  065013 (35pp).

\end{thebibliography}
\end{document}